\pdfoutput=1
\RequirePackage{ifpdf}
\ifpdf % We are running pdfTeX in pdf mode
\documentclass[pdftex]{sigma}
\else
\documentclass{sigma}
\fi

\usepackage{multirow}
\usepackage{diagbox}
\usepackage{hhline}
\usepackage[all]{xy}

\newcommand\tstrut{\rule{0pt}{2.6ex}}
\newcommand{\A}{\mbox{$\widehat A$}}

\newcommand{\Z}{\mbox{$\mathbb Z$}}
\newcommand{\C}{\mbox{$\mathbb C$}}
\newcommand{\ext}{\raise1pt\hbox{$\textstyle\bigwedge$}}
\newcommand{\sym}{S}

\newcommand{\ind}{\mbox{\rm ind}}
\newcommand{\ch}{\mbox{\rm ch}}
\newcommand{\vol}{\mbox{\rm vol}}
\newcommand{\coker}{\mbox{\rm coker}}

\newcommand{\dirac}{/\kern-5pt\partial}
\newcommand{\End}{{\rm End}}

\numberwithin{equation}{section}

\newtheorem{Theorem}{Theorem}[section]
\newtheorem{Lemma}[Theorem]{Lemma}
 { \theoremstyle{definition}
\newtheorem{Definition}[Theorem]{Definition}
\newtheorem{Example}[Theorem]{Example}
\newtheorem{Remark}[Theorem]{Remark} }

\begin{document}

\allowdisplaybreaks

\newcommand{\arXivNumber}{1609.01509}

\renewcommand{\PaperNumber}{027}

\FirstPageHeading

\ShortArticleName{Rigidity and Vanishing Theorems for Almost Even-Clif\/ford Hermitian Manifolds}

\ArticleName{Rigidity and Vanishing Theorems\\ for Almost Even-Clif\/ford Hermitian Manifolds}

\Author{Ana Lucia GARCIA-PULIDO and Rafael HERRERA}
\AuthorNameForHeading{A.L.~Garcia-Pulido and R.~Herrera}
\Address{Centro de Investigaci\'on en Matem\'aticas, A.~P.~402, Guanajuato, Gto., C.P.~36000, M\'exico}
\Email{\href{mailto:lucia@cimat.mx}{lucia@cimat.mx}, \href{mailto:rherrera@cimat.mx}{rherrera@cimat.mx}}
\URLaddress{\url{https://sites.google.com/site/algarciapulido}}

\ArticleDates{Received October 10, 2016, in f\/inal form April 19, 2017; Published online April 23, 2017}

\Abstract{We prove the rigidity and vanishing of several indices of ``geometrically natural'' twisted Dirac operators on
almost even-Clif\/ford Hermitian manifolds admitting circle actions by automorphisms.}

\Keywords{almost even-Clif\/ford Hermitian manifolds; index of elliptic operator; twisted Dirac operators; circle action by automorphisms}

\Classification{53C10; 53C15; 53C25; 58J20; 57S15}

\section{Introduction}

There are two classical vanishing theorems for the $\A$-genus (the index of the Dirac operator) on Spin manifolds: the Lichnerowicz vanishing \cite{Lichnerowicz} which assumes a metric of positive scalar curvature, and the Atiyah--Hirzebruch vanishing~\cite{AH} which assumes smooth circle action. These vanishings can be seen and have been used frequently as obstructions to the existence of such metrics or actions. More vanishing theorems for the indices of Spin$^c$ Dirac operators were explored by Hattori~\cite{Hattori} on almost complex manifolds and Spin$^c$ manifolds with compatible circle actions, which have parallels on complex manifolds with ample line bundles (a positivity condition for certain curvature) as in the case of the Kodaira vanishing theorem. Vanishing theorems have also been proven for indices of twisted Dirac operators on compact quaternion-K\"ahler manifolds with positive scalar curvature~\cite{LS}, and for almost quaternion-Hermitian manifolds with isometric circle actions that preserve the almost quaternion-Hermitian structure~\cite{Herrera-Herrera}.

The vanishings of such indices on manifolds with isometric circle actions are instances of the rigidity of elliptic operators under such actions,
an important property in the context of elliptic genera~\cite{BT,Dessai-rigidity,Hi-genera, HiSl,Liu, Taubes,W1, W2}. In this paper, we prove the rigidity and vanishing of the indices of several ``geometrically natural'' twisted Dirac operators on almost even-Clif\/ford manifolds admitting circle actions by automorphisms, resembling those studied on almost quaternionic-Hermitian manifolds.

The note is organized as follows. In Section~\ref{sec: preliminaries}, we recall some material on Clif\/ford algebras, Spin groups and representations, maximal tori of classical Lie groups, almost even-Clif\/ford Hermitian manifolds and their structure groups. In Section \ref{sec: twisted spinor bundles}, we examine the weights of the Spin representation in terms of the weights of the aforementioned structure groups and explore which representations to use in the twisted Dirac operators. In Section~\ref{sec: dirac-operators}, we prove the vanishing Theorems~\ref{theo: vanishing 1},~\ref{theo: vanishing 2} and~\ref{theo: vanishing 3}, using the Atiyah--Singer f\/ixed point theorem.

\section{Preliminaries}\label{sec: preliminaries}

The material presented in this section can be consulted in \cite{Arizmendi-Garcia-Herrera, Brocker-tomDieck, Friedrich}.

\subsection{Clif\/ford algebra, spin group and representation}
Let ${\rm Cl}_n$ denote the $2^n$-dimensional real Clif\/ford algebra generated by the orthonormal vectors $e_1, e_2, \dots, e_n\in \mathbb{R}^n$ subject to
the relations
\begin{gather*}e_i e_j + e_j e_i = -2\delta_{ij},\end{gather*}
and $\mathbb{C}{\rm l}_n={\rm Cl}_n\otimes_{\mathbb{R}}\mathbb{C}$ its complexif\/ication. The even Clif\/ford subalgebra ${\rm Cl}_r^0$ is def\/ined as the invariant $(+1)$-subspace of the involution of ${\rm Cl}_r$ induced by the map $-{\rm Id}_{\mathbb{R}^r}$.

There exist algebra isomorphisms
\begin{gather}
\mathbb{C}{\rm l}_n\cong
 \begin{cases}
 \End\big(\mathbb{C}^{2^k}\big) & \text{if $n=2k$},\\
 \End\big(\mathbb{C}^{2^k}\big)\oplus\End\big(\mathbb{C}^{2^k}\big) & \text{if $n=2k+1$},
 \end{cases}\label{eq: isomorphisms Clifford algebras}
\end{gather}
and the space of (complex) spinors is def\/ined to be
\begin{gather*}\Delta_n:=\mathbb{C}^{2^k}=\underbrace{\mathbb{C}^2\otimes \dots \otimes \mathbb{C}^2}_{k\,\,{\rm times}}.\end{gather*}
The map
\begin{gather*}\kappa\colon \ \mathbb{C}{\rm l}_n \longrightarrow \End\big(\mathbb{C}^{2^k}\big)\end{gather*}
is def\/ined to be either the aforementioned isomorphism for $n$ even, or the isomorphism followed by the projection onto the f\/irst summand for~$n$ odd. In order to make $\kappa$ explicit, consider the following matrices
\begin{gather*}{\rm Id} = \left(\begin{matrix}
1 & 0\\
0 & 1
 \end{matrix}\right),\qquad
g_1 = \left(\begin{matrix}
i & 0\\
0 & -i
 \end{matrix}\right),\qquad
g_2 = \left(\begin{matrix}
0 & i\\
i & 0
 \end{matrix}\right),\qquad
T = \left(\begin{matrix}
0 & -i\\
i & 0
 \end{matrix}\right).
\end{gather*}
In terms of the generators $e_1, \dots, e_n$ of the Clif\/ford algebra, $\kappa$ can be described explicitly as follows
\begin{gather}
e_1 \mapsto {\rm Id}\otimes {\rm Id}\otimes \dots\otimes {\rm Id}\otimes {\rm Id}\otimes g_1,\nonumber\\
e_2 \mapsto {\rm Id}\otimes {\rm Id}\otimes \dots\otimes {\rm Id}\otimes {\rm Id}\otimes g_2,\nonumber\\
e_3 \mapsto {\rm Id}\otimes {\rm Id}\otimes \dots\otimes {\rm Id}\otimes g_1\otimes T,\nonumber\\
e_4 \mapsto {\rm Id}\otimes {\rm Id}\otimes \dots\otimes {\rm Id}\otimes g_2\otimes T,\label{eq: explicit isomorphisms Clifford algebras}\\
 \cdots\cdots\cdots\cdots\cdots\cdots\cdots\cdots\cdots\cdots\cdots\nonumber\\
e_{2k-1} \mapsto g_1\otimes T\otimes \dots\otimes T\otimes T\otimes T,\nonumber\\
e_{2k} \mapsto g_2\otimes T\otimes\dots\otimes T\otimes T\otimes T,\nonumber
\end{gather}
and, if $n=2k+1$,
\begin{gather*} e_{2k+1}\mapsto i T\otimes T\otimes\dots\otimes T\otimes T\otimes T.\end{gather*}
The vectors
\begin{gather*}u_{+1}={1\over \sqrt{2}}(1,-i)\qquad\text{and}\qquad u_{-1}={1\over \sqrt{2}}(1,i),\end{gather*}
form a unitary basis of $\mathbb{C}^2$ with respect to the standard Hermitian product. Thus,{\samepage
\begin{gather*}\mathcal{B}=\{u_{\varepsilon_1,\dots,\varepsilon_k}=u_{\varepsilon_1}\otimes\dots\otimes
u_{\varepsilon_k}\,|\, \varepsilon_j=\pm 1,\, j=1,\dots,k\},\end{gather*}
is a unitary basis of $\Delta_n=\mathbb{C}^{2^k}$ with respect to the naturally induced Hermitian product.}

The Spin group ${\rm Spin}(n)\subset {\rm Cl}_n$ is the subset
\begin{gather*}{\rm Spin}(n) =\big\{x_1x_2\cdots x_{2l-1}x_{2l}\,|\,x_j\in\mathbb{R}^n, \,
|x_j|=1,\, l\in\mathbb{N}\big\},\end{gather*}
endowed with the product of the Clif\/ford algebra. It is a Lie group and its Lie algebra is
\begin{gather*}\mathfrak{spin}(n)=\mbox{span}\{e_ie_j\,|\,1\leq i< j \leq n\}.\end{gather*}
The restriction of $\kappa$ to ${\rm Spin}(n)$ def\/ines the Lie group representation
\begin{gather*}\kappa_n:=\kappa|_{{\rm Spin}(n)}\colon \ {\rm Spin}(n)\longrightarrow {\rm GL}(\Delta_n),\end{gather*}
which is, in fact, special unitary. We have the corresponding Lie algebra representation
\begin{gather*}\kappa_{n_*}\colon \ \mathfrak{spin}(n)\longrightarrow \mathfrak{gl}(\Delta_n).\end{gather*}
Recall that the Spin group ${\rm Spin}(n)$ is the universal double cover of ${\rm SO}(n)$, $n\ge 3$. For $n=2$ we consider ${\rm Spin}(2)$ to be the connected double cover of ${\rm SO}(2)$. The covering map will be denoted by
\begin{gather*}\lambda_n\colon \ {\rm Spin}(n)\rightarrow {\rm SO}(n)\subset {\rm GL}\big(\mathbb{R}^n\big).\end{gather*}
Its dif\/ferential is given by $\lambda_{n_*}(e_ie_j) = 2E_{ij}$, where $E_{ij}=e_i^*\otimes e_j - e_j^*\otimes e_i$ is the standard basis of the skew-symmetric matrices, and $e^*$ denotes the metric dual of the vector~$e$. Furthermore, we will abuse the notation and also denote by $\lambda_n$ the induced representation on the exterior algebra~$\ext^*\mathbb{R}^n$.

By means of $\kappa$, we have the Clif\/ford multiplication
\begin{align*}
\mu_n\colon \ \mathbb{R}^n\otimes \Delta_n &\longrightarrow\Delta_n, \\
x \otimes \phi &\longmapsto \mu_n(x\otimes \phi)=x\cdot\phi :=\kappa(x)(\phi) .
\end{align*}
The Clif\/ford multiplication $\mu_n$ is skew-symmetric with respect to the Hermitian product
\begin{gather*}
\langle x\cdot\phi_1 , \phi_2\rangle =\langle \mu_n(x\otimes \phi_1) , \phi_2\rangle =-\langle \phi_1 , \mu_n(x\otimes \phi_2)\rangle =-\langle \phi_1 , x\cdot \phi_2\rangle,
\end{gather*}
is ${\rm Spin}(n)$-equivariant and can be extended to a ${\rm Spin}(n)$-equivariant map
\begin{align*}
\mu_n\colon \ \ext^*\big(\mathbb{R}^n\big)\otimes \Delta_n &\longrightarrow\Delta_n,\\
\omega \otimes \psi &\longmapsto \omega\cdot\psi.
\end{align*}

When $n$ is even, we def\/ine the following involution
\begin{align*}
\Delta_n&\longrightarrow \Delta_n, \\
\psi &\longmapsto (-i)^{n\over 2}{\rm vol}_n\cdot \psi,
\end{align*}
where $\vol_n = e_1\cdots e_n$. The $\pm 1$ eigenspace of this involution is denoted $\Delta_n^\pm$. These spaces have equal dimension and are irreducible representations of ${\rm Spin}(n)$. Note that our def\/inition dif\/fers from the one given in \cite{Friedrich} by a $(-1)^{n\over 2}$. The reason for this dif\/ference is that we want the spinor~$u_{1,\dots,1}$ to be always positive. In this case, we will denote the two representations by
\begin{gather*}
\kappa_n^\pm\colon \ {\rm Spin}(n) \longrightarrow {\rm GL}\big(\Delta_n^\pm\big).
\end{gather*}

For future use, let us recall the ef\/fect of ${\rm vol}_n$ on $\Delta_n= \Delta_n^+ \oplus \Delta_n^-$ when $n$ is even:
\begin{gather*}\begin{array}{|c|c|c|}\hline
n\,\,\text{(mod 8)} & \Delta_n^+ & \Delta_n^- \tstrut\\\hline
0 & 1 & -1\\\hline
2 & i & -i\\\hline
4 & -1 & 1\\\hline
6 & -i & i \\\hline
 \end{array}
\end{gather*}
Furthermore, for $n\equiv 0$ (mod 4), $n\not =4$,
\begin{gather*}\ker(\kappa_n^+)= \begin{cases}
\{1,{\rm vol}_r\} & \text{if $r\equiv 0$ (mod 8)},\\
\{1,-{\rm vol}_r\} & \text{if $r\equiv 4$ (mod 8)},
 \end{cases}\end{gather*}
and
\begin{gather*}\ker(\kappa_n^-)= \begin{cases}
\{1,-{\rm vol}_r\} & \text{if $r\equiv 0$ (mod 8)},\\
\{1,{\rm vol}_r\} & \text{if $r\equiv 4$ (mod 8)}.
 \end{cases}\end{gather*}
For $r$ even, let
\begin{gather*}
\mathbb{P}{\rm SO}(r):= {{\rm SO}(r)\over\{\pm {\rm Id}_r\}}\cong {{\rm Spin}(r)\over\{\pm 1,\pm{\rm vol}_r\}},
\end{gather*}
and for $r\equiv 0$ (mod 4) let
\begin{gather*}
{\rm Spin}_\pm(r) \cong {{\rm Spin}(r)\over \{1, \pm{\rm vol}_r\} }.
\end{gather*}
Note that we will always denote by $1$ and ${\rm Id}_r$ the identity elements of ${\rm Spin}(r)$ and ${\rm SO}(r)$ respectively.

\subsection{Maximal tori}

\subsubsection[${\rm SO}(n)$]{$\boldsymbol{{\rm SO}(n)}$}

Recall that a maximal torus of ${\rm SO}(n)$ is given by
\begin{gather*}
 \left(
\begin{matrix}
\cos(\eta_1) & -\sin(\eta_1) & & & \\
\sin(\eta_1) & \cos(\eta_1) & & & \\
 & & \ddots & & \\
 & & & \cos(\eta_{n/2}) & -\sin(\eta_{n/2}) \\
 & & & \sin(\eta_{n/2}) & \cos(\eta_{n/2})
\end{matrix}
\right)
\end{gather*}
if $n$ is even, and
\begin{gather*}
 \left(
\begin{matrix}
\cos(\eta_1) & -\sin(\eta_1) & & & & \\
\sin(\eta_1) & \cos(\eta_1) & & & & \\
 & & \ddots & & &\\
 & & & \cos(\eta_{[n/2]}) & -\sin(\eta_{[n/2]}) & \\
 & & & \sin(\eta_{[n/2]}) & \cos(\eta_{[n/2]}) & \\
 & & & & & 1
\end{matrix}
\right)
\end{gather*}
if $n$ is odd.

\subsubsection[${\rm Spin}(n)$]{$\boldsymbol{{\rm Spin}(n)}$}
Each one of the $2\times 2$ rotation blocks is a transformation that can be achieved by using Clif\/ford product. For instance the rotation
\begin{gather*}\left(
\begin{matrix}
\cos(\varphi_1) & -\sin(\varphi_1) & & & \\
\sin(\varphi_1) & \cos(\varphi_1) & & & \\
 & & 1 & & \\
 & & & \ddots & \\
 & & & & 1
\end{matrix}
\right)
\end{gather*}
can be achieved by using the element
\begin{gather*}e_1(-\cos(\varphi_1/2)e_1 + \sin(\varphi_1/2)e_2)=\cos(\varphi_1/2) + \sin(\varphi_1/2)e_1e_2 \in {\rm Spin}(n)\end{gather*}
as follows
\begin{gather*}
(\cos(\varphi_1/2) + \sin(\varphi_1/2)e_1e_2)y(\cos(\varphi_1/2) + \sin(\varphi_1/2)e_2e_1) \\
\qquad{} = (y_1\cos(\varphi_1)-y_2\sin(\varphi_1))e_1 + (y_1\sin(\varphi_1)+y_2\cos(\varphi_1))e_2+y_3e_3+\cdots+y_ne_n,
\end{gather*}
for $y=y_1e_1+\cdots+y_ne_n\in\mathbb{R}^n$. Thus, we see that the corresponding elements
in ${\rm Spin}(n)$ are
\begin{gather*}\pm(\cos(\varphi_1/2) + \sin(\varphi_1/2)e_1e_2).\end{gather*}
Furthermore, we see that a maximal torus of ${\rm Spin}(n)$ is given by elements of the form
\begin{gather*}t(\varphi_1,\dots,\varphi_{k})=\prod_{j=1}^{k} (\cos(\varphi_j/2) + \sin(\varphi_j/2)e_{2j-1}e_{2j}).\end{gather*}
By using the explicit description~\eqref{eq: explicit isomorphisms Clifford algebras} of the isomorphisms \eqref{eq: isomorphisms Clifford algebras}, we can check that
\begin{gather*}t(\varphi_1,\dots,\varphi_{k})\cdot u_{\varepsilon_1,\dots,\varepsilon_k} =e^{{i\over 2}{\sum\limits_{j=1}^{k}\varepsilon_{k+1-j}\varphi_j}}\cdot u_{\varepsilon_1,\dots,\varepsilon_k},\end{gather*}
i.e., the basis vectors $u_{\varepsilon_1,\dots,\varepsilon_k}$ are weight vectors of the spin representation with weight
\begin{gather*}{1\over 2}\sum_{j=1}^{k}\varepsilon_{k+1-j}\varphi_j,\end{gather*}
which in coordinate vectors with respect to the basis $\{\varphi_j\}$ give the well known expressions
\begin{gather*}\left(\pm{1\over 2},\pm{1\over 2},\dots,\pm{1\over 2}\right).\end{gather*}
Indeed, in terms of the (appropriately ordered) basis $\mathcal{B}$, the matrix associated to an element $t(\varphi_1,\dots,\varphi_{[{n\over 2}]})$ is
\begin{gather*}
\left(\begin{array}{@{}l@{}}
e^{{i\over 2}(\varphi_1+\varphi_2+\cdots+\varphi_{[{n\over 2}]})} \\
\qquad e^{{i\over 2}(-\varphi_1+\varphi_2+\cdots+\varphi_{[{n\over 2}]})} \\
\qquad\qquad e^{{i\over 2}(\varphi_1-\varphi_2+\cdots+\varphi_{[{n\over 2}]})} \\
\qquad\qquad\qquad \ddots \\
\qquad\qquad\qquad\qquad e^{{i\over 2}(-\varphi_1-\varphi_2+\cdots+\varphi_{[{n\over 2}]})} \\
\qquad\qquad\qquad\qquad\qquad \ddots \\
\qquad\qquad\qquad\qquad\qquad\qquad e^{{i\over 2}(-\varphi_1-\varphi_2-\cdots-\varphi_{[{n\over 2}]})}
\end{array}\right).
\end{gather*}
Note that, when $n$ is even, $\Delta_n^+$ is generated by the basis vectors $u_{\varepsilon_1,\dots,\varepsilon_{{n\over 2}}}$ with an even number
of $\varepsilon_j$ equal to $-1$, and $\Delta_n^-$ is generated by the basis vectors $u_{\varepsilon_1,\dots,\varepsilon_{{n\over 2}}}$ with an odd number of $\varepsilon_j$ equal to $-1$. Therefore, after reordering the basis, the matrix above can be rearranged to have two diagonal blocks of equal size: one block in which the exponents contain an even number of negative signs
\begin{gather*}
 \left(\begin{array}{@{}l@{}}
e^{{i\over 2}(\varphi_1+\varphi_2+\cdots+\varphi_{{n\over 2}})} \\
\qquad e^{{i\over 2}(-\varphi_1-\varphi_2+\cdots+\varphi_{{n\over 2}})} \\
\qquad\qquad e^{{i\over 2}(-\varphi_1+\varphi_2-\cdots+\varphi_{{n\over 2}})} \\
\qquad\qquad\qquad \ddots
 \end{array}\right),
\end{gather*}
and another block in which the exponents contain an odd number of negative signs
\begin{gather*}
\left(\begin{array}{@{}l@{}}
e^{{i\over 2}(-\varphi_1+\varphi_2+\cdots+\varphi_{{n\over 2}})} \\
\qquad e^{{i\over 2}(\varphi_1-\varphi_2+\cdots+\varphi_{{n\over 2}})} \\
\qquad\qquad e^{{i\over 2}(\varphi_1+\varphi_2-\varphi_3+\cdots+\varphi_{{n\over 2}})} \\
\qquad\qquad\qquad \ddots
 \end{array}\right).
\end{gather*}

\subsubsection[${\rm U}(m)$]{$\boldsymbol{{\rm U}(m)}$}

The standard maximal torus of ${\rm U}(m)$ is
\begin{gather*}\left(
\begin{matrix}
e^{i\theta_1} & & & \\
 & e^{i\theta_2} & & \\
 & & \ddots & \\
 & & & e^{i\theta_m}\\
 & & &
\end{matrix}
\right).
\end{gather*}

\subsubsection{${\rm Sp}(m)$}

The standard maximal torus of ${\rm Sp}(m)$ is
\begin{gather*}\left(
\begin{matrix}
e^{i\theta_1} & & & & \\
 & e^{-i\theta_1} & & & \\
 & & \ddots & & \\
 & & & e^{i\theta_m} & \\
 & & & & e^{-i\theta_m}
\end{matrix}
\right).
\end{gather*}

\subsection{Almost even-Clif\/ford Hermitian structures}

\begin{Definition} Let $N\in \mathbb{N}$ and $(e_1,\dots,e_r)$ an orthonormal frame of $\mathbb{R}^r$.
\begin{itemize}\itemsep=0pt
\item A {\em linear even-Clifford structure of rank $r$} on $\mathbb{R}^N$ is an algebra representation
\begin{gather*}\Phi\colon \ {\rm Cl}_r^0\longrightarrow \End\big(\mathbb{R}^N\big).\end{gather*}
\item A {\em linear even-Clifford Hermitian structure of rank~$r$} on $\mathbb{R}^N$ (endowed with a positive def\/inite inner product) is a linear even-Clif\/ford structure of rank~$r$ such that each bivec\-tor~$e_ie_j$, $1\leq i< j \leq r$, is mapped to a skew-symmetric endomorphism $\Phi(e_ie_j)=J_{ij}$.
\end{itemize}
\end{Definition}

\begin{Remark}\quad
\begin{itemize}\itemsep=0pt
\item Note that $J_{ij}^2=-{\rm Id}_{\mathbb{R}^N}$.
\item Given a linear even-Clif\/ford structure of rank $r$ on $\mathbb{R}^N$, we can average the standard inner product $\langle \,,\,\rangle $ on $\mathbb{R}^N$ as follows
\begin{gather*}(X,Y)=\sum_{k=1}^{[r/2]} \bigg[\sum_{1\leq i_1<\dots<i_{2k}<r}
\langle \Phi(e_{i_1\dots i_{2k}})(X),\Phi(e_{i_1\dots i_{2k}})(Y)\rangle \bigg], \end{gather*}
where $(e_1,\dots,e_r)$ is an orthonormal frame of $\mathbb{R}^r$, so that the linear even-Clif\/ford structure is Hermitian with respect to the averaged inner product.
\item Given a linear even-Clif\/ford Hermitian structure of rank $r$, the subalgebra~$\mathfrak{spin}(r)$ is mapped injectively into the skew-symmetric endomorphisms $\End^-(\mathbb{R}^N)$.
\end{itemize}
\end{Remark}

\begin{Definition} Let $r\geq 2$.
\begin{itemize}\itemsep=0pt
 \item A {\em rank $r$ almost even-Clifford structure} on a smooth manifold $M$ is a smoothly varying choice of a rank $r$ linear even-Clif\/ford structure on each tangent space of~$M$.

\item A smooth manifold carrying an almost even-Clif\/ford structure will be called an {\em almost even-Clifford manifold}.

\item A {\em rank $r$ almost even-Clifford Hermitian structure} on a Riemannian manifold~$M$ is \linebreak a~smoothly varying choice of a~linear even-Clif\/ford Hermitian structure on each tangent space of~$M$.

\item A Riemannian manifold carrying a rank $r$ almost even-Clif\/ford Hermitian structure will be called a~{\em rank~$r$ almost even-Clifford Hermitian manifold}, or an {\em almost-${\rm Cl}_r^0$-Hermitian manifold} for short.
\end{itemize}
\end{Definition}

\begin{Remark}
Our def\/inition of almost even-Clif\/ford Hermitian structure does not require the existence of a Riemannian vector bundle of rank $r$. Therefore, it includes both the notions of even Clif\/ford structure and projective even Clif\/ford structure introduced in \cite[Def\/inition~2.2 and Remark~2.5]{Moroianu-Semmelmann}.
\end{Remark}

\subsubsection{Structure groups of almost even-Clif\/ford manifolds}\label{subsubsection: structure groups}

Thanks to \cite{Arizmendi-Herrera}, we know that the complexif\/ication of the tangent space of an almost-${\rm Cl}_r^0$-Hermitian manifold decomposes as follows
\begin{gather}
\begin{array}{|c|c|}
 \hline
 r \mbox{\ {\rm (mod 8)} } &\mathbb{R}^N\otimes \mathbb{C}
\rule{0pt}{3ex}\\
\hline
 0 & \mathbb{C}^{m_1}\otimes \Delta_r^+ \oplus \mathbb{C}^{m_2}\otimes \Delta_r^-\tstrut\\
 \hline
 1,\,7 & \mathbb{C}^m\otimes \Delta_r\tstrut\\
\hline
 2 & \mathbb{C}^m\otimes \Delta_r^+\oplus\overline{\mathbb{C}^m}\otimes\Delta_r^-\tstrut\\
\hline
 6 & \overline{\mathbb{C}^m}\otimes \Delta_r^+\oplus \mathbb{C}^m\otimes\Delta_r^-\tstrut\\
\hline
 3,\,5 & \mathbb{C}^{2m}\otimes \Delta_r\tstrut\\
\hline
 4 & \mathbb{C}^{2m_2}\otimes \Delta_r^+ \oplus \mathbb{C}^{2m_1}\otimes \Delta_r^-\tstrut\\
\hline
 \end{array}\label{eq: complexified R^N}
\end{gather}
where the dif\/ferent $\mathbb{C}^p$ denote the corresponding standard complex representations of the classical Lie groups ${\rm SO}(p)$, ${\rm U}(p)$ or ${\rm Sp}(p)$. Note that the dimension of an almost even-Clif\/ford Hermitian manifold depends of two or three parameters: the rank $r$ of the even-Clif\/ford structure and the multiplicity~$m$ or multiplicities $m_1$, $m_2$.

The structure groups of the aforementioned manifolds, for $r\geq 3$, are given as follows (see~\cite{Arizmendi-Garcia-Herrera}):
\begin{itemize}\itemsep=0pt
 \item For $r\not \equiv 0$ (mod 4)
\begin{center}
\begin{tabular}{|c|c|c|}\hline
\backslashbox{$r$ (mod 8)}{$m$ (mod 2)} & 0 & 1 \rule{0pt}{3ex} \\\hline
1, 7 & ${{\rm SO}(m)\times {\rm Spin}(r)\over \{\pm({\rm Id}_{m},1)\}}$ & ${\rm SO}(m)\times {\rm Spin}(r)$ \tsep{5pt}\bsep{5pt} \\\hline
2, 6 & \multicolumn{2}{c|}{${{\rm U}(m)\times {\rm Spin}(r)\over \{\pm({\rm Id}_m,1),\pm(i{\rm Id}_m,-{\rm vol}_r)\}}$} \tsep{5pt}\bsep{5pt}\\\hline
3, 5 & \multicolumn{2}{c|}{${{\rm Sp}(m)\times {\rm Spin}(r)\over \{\pm({\rm Id}_m,1)\}}$} \tsep{5pt}\bsep{5pt}\\\hline
\end{tabular}
\end{center}

\item
For $r\equiv 0$ (mod 8)
\begin{center}
\scalebox{0.85}{
\begin{tabular}{|c|c|c|c|}\hline
\backslashbox{$m_1$}{$m_2$} & 0 & 0 (mod 2) & 1 (mod 2) \tstrut\\\hline
0 & & ${{\rm SO}(m_2)\times {\rm Spin}(r)\over \{\pm({\rm Id}_{m_2},1),\pm({\rm Id}_{m_2},-{\rm vol}_r)\}}$ & ${{\rm SO}(m_2)\times {\rm Spin}(r)\over \langle ({\rm Id}_{m_2},-{\rm vol}_r)\rangle }$\tsep{5pt}\bsep{5pt} \\\hline
0 (mod 2)& ${{\rm SO}(m_1)\times {\rm Spin}(r)\over \{\pm({\rm Id}_{m_1},1),\pm({\rm Id}_{m_1},{\rm vol}_r)\}}$ & ${{\rm SO}(m_1)\times {\rm SO}(m_2)\times {\rm Spin}(r)\over \{\pm({\rm Id}_{m_1},{\rm Id}_{m_2},1),\pm({\rm Id}_{m_1},-{\rm Id}_{m_2},{\rm vol}_r)\}}$ & ${{\rm SO}(m_1)\times {\rm SO}(m_2)\times {\rm Spin}(r)\over \langle (-{\rm Id}_{m_1},{\rm Id}_{m_2},-{\rm vol}_r)\rangle }$\tsep{5pt}\bsep{5pt} \\\hline
1 (mod 2)& ${{\rm SO}(m_1)\times {\rm Spin}(r)\over \langle ({\rm Id}_{m_1},{\rm vol}_r)\rangle }$ & ${{\rm SO}(m_1)\times {\rm SO}(m_2)\times {\rm Spin}(r)\over \langle ({\rm Id}_{m_1},-{\rm Id}_{m_2},{\rm vol}_r)\rangle }$ & ${\rm SO}(m_1)\times {\rm SO}(m_2)\times {\rm Spin}(r)$ \tsep{5pt}\bsep{5pt}\\\hline
\end{tabular}}
\end{center}

\item
For $r\equiv 4$ (mod 8)
\begin{center}
\scalebox{0.9}{
 \begin{tabular}{|c|c|c|c|}\hline
%\backslashbox{$r\equiv 4$ (mod 8)}{$m_1,m_2$}
& $m_1,m_2>0$ & $m_1>0$, $m_2=0$ & $m_1=0$, $m_2>0$ \rule{0pt}{3ex} \\\hline
$r=4$ & \multirow{2}{*}{${{\rm Sp}(m_1)\times {\rm Sp}(m_2)\times {\rm Spin}(r)\over \{\pm({\rm Id}_{2m_1},{\rm Id}_{2m_2},1),\pm({\rm Id}_{2m_1},-{\rm Id}_{2m_2},{\rm vol}_r)\}}$} &
 ${{\rm Sp}(m_1)\times {\rm Spin}(3)\over \{\pm({\rm Id}_{2m_1},1)\}}$ & ${{\rm Sp}(m_2)\times {\rm Spin}(3)\over \{\pm({\rm Id}_{2m_2},1)\}}$ \tsep{5pt}\bsep{5pt} \\\hhline{-~--}
$r>4$ & & ${{\rm Sp}(m_1)\times {\rm Spin}(r)\over \{\pm({\rm Id}_{2m_1},1),\pm({\rm Id}_{2m_1},{\rm vol}_r)\}}$ & ${{\rm Sp}(m_2)\times {\rm Spin}(r)\over \{\pm({\rm Id}_{2m_2},1),\pm({\rm Id}_{2m_2},-{\rm vol}_r)\}}$ \tsep{5pt}\bsep{5pt} \\\hline
\end{tabular}}
\end{center}
\end{itemize}

Note that for $r=2$, the structure group is actually ${\rm U}(m)$.

Since all of these groups are quotients of products $G\times {\rm Spin}(r)$, where $G$ is a (product of) classical Lie group(s), it will be useful to know if they can be mapped to either ${\rm Spin}(r)$, or ${\rm SO}(r)$ or $\mathbb{P}{\rm SO}(r)$. It is easy to see that they map as follows
\begin{itemize}\itemsep=0pt
 \item
 For $r\not \equiv 0$ (mod 4)
\begin{center}
\begin{tabular}{|c|c|c|}\hline
\backslashbox{$r$ (mod 8)}{$m$ (mod 2)} & 0 & 1 \rule{0pt}{3ex} \\\hline
1, 7 & ${\rm SO}(r)$ & ${\rm Spin}(r)$ \tstrut\\\hline
2, 6 & \multicolumn{2}{c|}{$\mathbb{P}{\rm SO}(r)$} \tstrut\\\hline
3, 5 & \multicolumn{2}{c|}{${\rm SO}(r)$} \tstrut\\\hline
\end{tabular}
\end{center}
\item
For $r\equiv 0$ (mod 8)
\begin{center}
\begin{tabular}{|c|c|c|}\hline
\backslashbox{$m_1$}{$m_2$} & 0 (mod 2) & 1 (mod 2) \tstrut\\\hline
0 (mod 2)& $\mathbb{P}{\rm SO}(r)$ & ${\rm Spin}_-(r)$\tstrut\\\hline
1 (mod 2)& ${\rm Spin}_+(r)$ & ${\rm Spin}(r)$ \tstrut\\\hline
\end{tabular}
\end{center}
\item
For $r\equiv 4$ (mod 8){\samepage
\begin{center}
\begin{tabular}{|c|c|c|c|}\hline
\backslashbox{$r\equiv 4$ (mod 8)}{$m_1$, $m_2$} & $m_1,m_2>0$ & $m_1>0$, $m_2=0$ & $m_1=0$, $m_2>0$ \rule{0pt}{3ex} \\\hline
$r=4$ & \multirow{2}{*}{$\mathbb{P}{\rm SO}(r)$} &
 ${\rm SO}(3)$ & ${\rm SO}(3)$ \tstrut\\\hhline{-~--}
$r>4$ & & $\mathbb{P}{\rm SO}(r)$ & $\mathbb{P}{\rm SO}(r)$ \tstrut\\\hline
\end{tabular}
\end{center}}
\end{itemize}
This can be summarized roughly as follows: the structure group of an almost-${\rm Cl}_r^0$-Hermitian manifold of rank $r$ maps to ${\rm SO}(r)$ if $r$ is odd, and maps to $\mathbb{P}{\rm SO}(r)$ if $r$ is even.

For future use, we will establish the notation for the decomposition of the complexif\/ied tangent bundle of an almost-${\rm Cl}_r^0$-Hermitian manifold:
\begin{gather}
\begin{array}{|c|c|}
 \hline
 r \mbox{\ {\rm (mod 8)} } &TM\otimes \mathbb{C}
\rule{0pt}{3ex}\\
\hline
 0 & E_1\otimes \Delta_r^+ \oplus E_2\otimes \Delta_r^-\tstrut\\
 \hline
 1,\, 7 & E\otimes \Delta_r\tstrut\\
\hline
 2 & E\otimes \Delta_r^+\oplus\overline{E}\otimes\Delta_r^-\tstrut\\
\hline
 6 & \overline{E}\otimes \Delta_r^+\oplus E\otimes\Delta_r^-\tstrut\\
\hline
 3,\, 5 & E\otimes \Delta_r\tstrut\\
\hline
 4 & E_2\otimes \Delta_r^+ \oplus E_1\otimes \Delta_r^-\tstrut\\
\hline
 \end{array} \label{eq: complexified tangent space}
\end{gather}
where $E$, $E_1$, $E_2$ are locally def\/ined vector bundles with f\/ibre $\mathbb{C}^p$ which correspond to the standard complex representation
of the dif\/ferent Lie groups mentioned in %Table
\eqref{eq: complexified R^N}.

\subsection{A useful lemma}

\begin{Lemma}\label{lemma: useful lemma}
Let $x\in \mathbb{C}$ and $k,m\in \mathbb{Z}/2$ such that $k+m\in \mathbb{Z}$. If $|k|<|m|$, then
\begin{gather*}G(z)={z^{k}\over z^{-m}e^{x} - z^{m}e^{-x}}\end{gather*}
 is a rational function on $\mathbb{C}$ and
\begin{gather*}\lim_{z\rightarrow 0 } G(z) = 0 = \lim_{z\rightarrow \infty } G(z).\end{gather*}
\end{Lemma}

\section[Twisted spinor bundles on almost-${\rm Cl}_r^0$-Hermitian manifolds]{Twisted spinor bundles on almost-$\boldsymbol{{\rm Cl}_r^0}$-Hermitian manifolds}\label{sec: twisted spinor bundles}

In this subsection, we present some calculations relevant to the global def\/initon of twisted spinor bundles.

When the structure group of an oriented $N$-dimensional Riemannian manifold reduces to a~proper subgroup $G\subset {\rm SO}(N)$, one can associate vector bundles to the corresponding $G$-principal bundle~$P_G$ by means of the representations of~$G$. If the manifold is Spin, one can ask if there exists a lifting map $\tilde{i}$ making the following diagram commute
\begin{gather*}
\xymatrix{
 & {\rm Spin}(N) \ar[d]^{2:1}\\
G\ar[r]^i \ar[ur]^{\tilde{i}} & {\rm SO}(N)
}
\end{gather*}
in which case, the Spin representation $\Delta_N$ may decompose according to $G$.

Even when such map $\tilde{i}$ does not exist (necessarily $\pi_1(G)\not= \{1\}$), there may be a f\/inite covering space $G'$ of $G=G'/\Gamma$ for which it does, and one can then decompose the Spin representation according to $G'$. We can now check how the elements of the f\/inite subgroup $\Gamma$ act on $\Delta_N$, and at least some of them will act non-trivially, thus conf\/irming that there cannot be a map $\tilde{i}$. By observing this action, we can then consider tensoring $\Delta_N$ with another representation $V$ of $G'$ such that $\Gamma$ now acts trivially on $\Delta_N\otimes V$.

In the context of almost-${\rm Cl}_r^0$-Hermitian manifolds, the structure group embeds into the relevant Spin group \cite[Theorem~4.1]{Arizmendi-Garcia-Herrera}, with the exception of four cases which we will analyze. More precisely, we found that
\begin{itemize}\itemsep=0pt
 \item ${{\rm Sp}(m)\times {\rm Spin}(3)\over \{\pm ({\rm Id}_{2m},1)\}}$ does not embed into ${\rm Spin}(4m)$ if $m$ is odd;
 \item ${{\rm Sp}(m_1)\times {\rm Sp}(m_2)\times {\rm Spin}(4)\over \{\pm ({\rm Id}_{2m_1},{\rm Id}_{2m_2},1),\pm ({\rm Id}_{2m_1},-{\rm Id}_{2m_2},{\rm vol}_4)\}}$ does not embed into ${\rm Spin}(4(m_1+m_2))$ if either~$m_1$ or~$m_2$ (or both) are odd;
 \item ${{\rm U}(m)\times {\rm Spin}(6)\over\{\pm ({\rm Id}_m,1),(i{\rm Id}_m, -{\rm vol}_6)\}}$ does not embed into
 ${\rm Spin}(8m)$ if $m$ is odd;
 \item ${{\rm SO}(m_1)\times {\rm SO}(m_2)\times {\rm Spin}(8)\over \{({\rm Id}_{m_1},{\rm Id}_{m_2},1),({\rm Id}_{m_1},-{\rm Id}_{2m_2},{\rm vol}_8)\}}$ if $m_1+1\equiv m_2\equiv 0$ (mod 2), % and\\
 ${{\rm SO}(m_1)\times {\rm SO}(m_2)\times {\rm Spin}(8)\over \{({\rm Id}_{m_1},{\rm Id}_{m_2},1),(-{\rm Id}_{m_1},{\rm Id}_{2m_2},-{\rm vol}_8)\}}$ if $m_1\equiv m_2+1\equiv 0$ (mod 2)
 do not embed into ${\rm Spin}(8(m_1+m_2))$.
\end{itemize}
However, by the same calculations in \cite{Arizmendi-Garcia-Herrera} we know that there are homomorphisms
\begin{itemize}\itemsep=0pt
 \item ${\rm Sp}(m)\times {\rm Spin}(3)\longrightarrow {\rm Spin}(4m)$;
\item ${\rm Sp}(m_1)\times {\rm Sp}(m_2)\times {\rm Spin}(4)\longrightarrow {\rm Spin}(4(m_1+m_2))$;
\item ${\rm U}(m)\times {\rm Spin}(6)\longrightarrow {\rm Spin}(8m)$;
\item ${\rm SO}(m_1)\times {\rm SO}(m_2)\times {\rm Spin}(8)\longrightarrow {\rm Spin}(8(m_1+m_2))$.
\end{itemize}

In order to analyze this situation and the appropriate twisting bundles for almost-${\rm Cl}_r^0$-Hermitian manifolds in general, we need to set up some notation regarding weights of Lie groups.

\subsection[Weights of ${\rm SO}(N)$ with respect to the structure subgroups]{Weights of $\boldsymbol{{\rm SO}(N)}$ with respect to the structure subgroups}\label{subsec: weights}

We need to rewrite the weights of ${\rm SO}(N)$ in terms of the maximal torus of the relevant structure group. Let $(\eta_1,\dots, \eta_{N/2})$ denote the coordinates of a maximal torus of ${\rm SO}(N)$, and $(\varphi_1,\dots,\varphi_{[{r\over 2}]})$ denote the coordinates of a maximal torus of ${\rm SO}(r)$.

For $r$ odd, let $\lambda_1,\dots,\lambda_{2^{[{r\over 2}]}}$ denote the weights of $\Delta_r$
\begin{gather*}\pm {1\over 2}\varphi_1\pm\cdots\pm {1\over 2}\varphi_{[{r\over 2}]},\end{gather*}
listed in some order such that the f\/irst half of weights have an even number of negative signs, and the second half of weights have an odd number of negative signs.

For $r$ even, let $\lambda_1^\pm,\dots,\lambda_{2^{{r\over 2}-1}}^\pm$ denote the weights of $\Delta_r^\pm$
\begin{gather*}\pm {1\over 2}\varphi_1\pm\cdots\pm {1\over 2}\varphi_{{r\over 2}},\end{gather*}
which have an even and odd number of negative signs respectively. If $r \equiv 0 \mod 4$ we will be considering $\lambda_1^\pm,\dots,\lambda_{2^{{r\over 2}-1}}^\pm$ to be listed in some order so that the f\/irst and second halves are interchanged by ref\/lection (changing all the signs),

\subsubsection[$r\equiv1,\, 7 \mod 8$]{$\boldsymbol{r\equiv1,\,7 \mod 8}$}
Let $(\theta_1,\dots,\theta_{[{m\over2}]})$ denote the coordinates of maximal tori of ${\rm SO}(m)$. Since
\begin{gather*}\mathbb{C}^N=\mathbb{C}^{m}\otimes \Delta_r,\end{gather*}
we can set
\begin{gather*}
 \eta_{(j-1)2^{[{r\over 2}]}+k}= \theta_j + \lambda_k
\end{gather*}
 if $m$ is even, and
\begin{gather*}
 \eta_{(j-1)2^{[{r\over 2}]}+k}= \theta_j + \lambda_k, \qquad
 \eta_{[{m\over2}]2^{[{r\over 2}]}+l}= \lambda_l
\end{gather*}
if $m$ is odd, where $1\leq j\leq [{m\over 2}]$, $1\leq k \leq 2^{[{r\over 2}]}$ and $1\leq l \leq 2^{[{r\over 2}]-1}$ in both cases.

\subsubsection[$r\equiv 2,\,6 \mod 8$]{$\boldsymbol{r\equiv 2,\,6 \mod 8}$}
Let $(\theta_1,\dots,\theta_m)$ denote the coordinates of maximal tori of ${\rm U}(m)$. Since
\begin{gather*}\mathbb{C}^N=
\begin{cases}
\mathbb{C}^{m}\otimes \Delta_r^+\oplus \overline{\mathbb{C}^{m}}\otimes \Delta_r^- & \text{if $r\equiv 2$ (mod 8)},\\
\overline{\mathbb{C}^{m}}\otimes \Delta_r^+\oplus \mathbb{C}^{m}\otimes \Delta_r^- & \text{if $r\equiv 6$ (mod 8)},
\end{cases}\end{gather*}
we can set
\begin{gather}
 \eta_{(j-1)2^{{r\over 2}-1}+k}=
 \begin{cases}
 \theta_j + \lambda_k^+ & \text{if $r\equiv 2$ (mod 8)},\\
 \theta_j + \lambda_k^- & \text{if $r\equiv 6$ (mod 8)},
 \end{cases}\label{eq: pesos r=2,6}
\end{gather}
where $1\leq j\leq m$ and $1\leq k \leq 2^{{r\over 2}-1}$.

\subsubsection[$r\equiv 3,\,5 \mod 8$]{$\boldsymbol{r\equiv 3,\,5 \mod 8}$}
Let $(\theta_1,\dots,\theta_m)$ denote the coordinates of maximal tori of ${\rm Sp}(m)$. Since
\begin{gather*}\mathbb{C}^N=\mathbb{C}^{2m}\otimes \Delta_r,\end{gather*}
we can set
\begin{gather}
 \eta_{(j-1)2^{[{r\over 2}]}+k}= \theta_j + \lambda_k , \label{eq: pesos r=3,5}
\end{gather}
where $1\leq j\leq m$ and $1\leq k \leq 2^{[{r\over 2}]}$.

\subsubsection[$r\equiv 4 \mod 8$]{$\boldsymbol{r\equiv 4 \mod 8}$}
Let $(\theta_1,\dots,\theta_{m_1})$ and $(\theta_1',\dots,\theta_{m_2}')$ denote the coordinates of maximal tori of ${\rm Sp}(m_1)$ and ${\rm Sp}(m_2)$ respectively. Since
\begin{gather*}\mathbb{C}^N=\mathbb{C}^{2m_1}\otimes \Delta_r^+\oplus\mathbb{C}^{2m_2}\otimes \Delta_r^-,\end{gather*}
we can set
\begin{gather}
 \eta_{(j_1-1)2^{{r\over 2}-1}+k}= \theta_{j_1} + \lambda_k^+, \qquad
 \eta_{m_12^{{r\over 2}-1}+(j_2-1)2^{{r\over 2}-1}+k}= \theta_{j_2}' + \lambda_k^- , \label{eq: pesos r=4}
\end{gather}
where $1\leq j_1\leq m_1$, $1\leq j_2\leq m_2$ and $1\leq k \leq 2^{{r\over 2}-1}$.

\subsubsection[$r\equiv 0 \mod 8$]{$\boldsymbol{r\equiv 0 \mod 8}$}
Let $(\theta_1,\dots,\theta_{[{m_1\over2}]})$ and $(\theta_1',\dots,\theta_{[{m_2\over2}]}')$ denote the coordinates of maximal tori of ${\rm SO}(m_1)$ and ${\rm SO}(m_2)$ respectively. Since
\begin{gather*}\mathbb{C}^N=\mathbb{C}^{m_1}\otimes \Delta_r^+\oplus\mathbb{C}^{m_2}\otimes \Delta_r^-,\end{gather*}
we can set
\begin{itemize}\itemsep=0pt
 \item if $m_1$, $m_2$ are even,
 \begin{gather*}
 \eta_{(j_1-1)2^{{r\over 2}-1}+k}= \theta_{j_1} + \lambda_k^+, \qquad
 \eta_{m_12^{{r\over 2}-1}+(j_2-1)2^{{r\over 2}-1}+k}= \theta_{j_2}' + \lambda_k^- ,
\end{gather*}
where $1\leq j_1\leq {m_1\over 2}$, $1\leq j_2\leq {m_2\over 2}$ and $1\leq k \leq 2^{{r\over 2}-1}$;

 \item if $m_1$ is even and $m_2$ is odd,
 \begin{gather}
 \eta_{(j_1-1)2^{{r\over 2}-1}+k}= \theta_{j_1} + \lambda_k^+, \qquad
 \eta_{m_12^{{r\over 2}-1}+(j_2-1)2^{{r\over 2}-1}+k}= \theta_{j_2}' + \lambda_k^-, \nonumber\\
 \eta_{m_12^{{r\over 2}-1}+[{m_2\over2}]2^{{r\over 2}-1}+l}= \lambda_l^- ,\label{eq: pesos r=8 m1 even m2 odd}
\end{gather}
where $1\leq j_1\leq {m_1\over 2}$, $1\leq j_2\leq [{m_2\over 2}]$, $1\leq k \leq 2^{{r\over 2}-1}$ and $1\leq l \leq 2^{{r\over 2}-2}$;

 \item if $m_1$ is odd and $m_2$ is even,
 \begin{gather}
 \eta_{(j_1-1)2^{{r\over 2}-1}+k}= \theta_{j_1} + \lambda_k^+, \qquad
 \eta_{[{m_1\over2}]2^{{r\over 2}-1}+l}= \lambda_l^+ ,\nonumber\\
 \eta_{[{m_1\over2}]2^{{r\over 2}-1}+2^{{r\over 2}-2}+(j_2-1)2^{{r\over 2}-1}+k}= \theta_{j_2}'+ \lambda_k^- ,\label{eq: pesos r=8 m1 odd m2 even}
\end{gather}
where $1\leq j_1\leq [{m_1\over 2}]$, $1\leq j_2\leq {m_2\over 2}$, $1\leq k \leq 2^{{r\over 2}-1}$ and $1\leq l \leq 2^{{r\over 2}-2}$;

 \item if $m_1$, $m_2$ are odd,
 \begin{gather*}
 \eta_{(j_1-1)2^{{r\over 2}-1}+k}= \theta_{j_1} + \lambda_k^+, \qquad \eta_{[{m_1\over2}]2^{{r\over 2}-1}+l}= \lambda_l^+ ,\\
 \eta_{[{m_1\over2}]2^{{r\over 2}-1}+2^{{r\over 2}-2}+(j_2-1)2^{{r\over 2}-1}+k}= \theta_{j_2}'+ \lambda_k^- ,\qquad
 \eta_{[{m_1\over2}]2^{{r\over 2}-1}+2^{{r\over 2}-2}+[{m_2\over 2}]2^{{r\over 2}-1}+l}= \lambda_l^- ,
\end{gather*}
where $1\leq j_1\leq [{m_1\over 2}]$, $1\leq j_2\leq [{m_2\over 2}]$, $1\leq k \leq 2^{{r\over 2}-1}$ and $1\leq l \leq 2^{{r\over 2}-2}$.
\end{itemize}

\subsection[The Spin representation when $r=3,4,6,8$]{The Spin representation when $\boldsymbol{r=3,\,4,\,6,\,8}$}\label{subsec: conditions 1}

The elements of the f\/inite subgroups involved in the structure groups of almost-${\rm Cl}_r^0$-Hermitian manifolds actually belong to maximal tori. Thus we can calculate their ef\/fect on representations in terms of the weights we just described. In this subsection, we examine the cases when the structure group does not embed into ${\rm Spin}(N)$.

\subsubsection[$r=3$]{$\boldsymbol{r=3}$}

Recall \eqref{eq: pesos r=3,5}, which in this case is
\begin{gather*}
 \eta_{2j-1} = \theta_j+{\varphi_1\over 2}, \qquad \eta_{2j}= \theta_j-{\varphi_1\over 2} ,
\end{gather*}
so that the weights of the spin representation are
\begin{gather*}
 \pm {\eta_1\over 2} \pm \cdots \pm {\eta_{2m}\over 2} =
 \sum_{j\in I_1} {\theta_j+{\varphi_1\over 2}\over 2} - \sum_{j\in \bar{I_1}} {\theta_j+{\varphi_1\over 2}\over 2}
 + \sum_{j\in I_2} {\theta_j-{\varphi_1\over 2}\over 2} - \sum_{j\in \bar{I_2}} {\theta_j-{\varphi_1\over 2}\over 2},
\end{gather*}
where $I_1,I_2\subseteq \{1,\dots,m\}$, and $\bar{I_j} = \{1,\dots,m\}-I_j$ denote their complements, $j = 1,2$. The element $(-{\rm Id}_{2m},-1)\in {\rm Sp}(m)\times {\rm Spin}(3)$ corresponds to the parameters
\begin{gather*}
 \theta_j = \pi , \qquad \varphi_1 = 2\pi,
\end{gather*}
for $1\leq j\leq m$, so that such a sum is equal to
\begin{gather*}2|I_1|\pi - m\pi\end{gather*}
and the ef\/fect of $(-{\rm Id}_{2m},-1)$ on each weight line is
\begin{gather*}e^{-im\pi} = (-1)^m.
\end{gather*}
Thus, $(-{\rm Id}_{2m},-1)\in {\rm Sp}(m)\times {\rm Spin}(3)$ acts trivially on $\Delta_{4m}$ if $m$ is even and as multiplication by~$(-1)$ if~$m$ is odd.

Thus, in order to have a twisted Spin representation
\begin{gather*}\Delta_{4m}\otimes \ext^u \mathbb{C}^{2m}\otimes (\Delta_3)^{\otimes s}\end{gather*}
of ${{\rm Sp}(m)\times {\rm Spin}(3)\over \{\pm ({\rm Id}_{2m},1)\}}$, the exponents must satisfy $m+u+s\equiv 0 \,\,\,\mbox{(mod 2)},$ which is a well known fact for almost quaternion-Hermitian manifolds~\cite{SalPitman}.

\subsubsection[$r=4$]{$\boldsymbol{r=4}$}

By \eqref{eq: pesos r=4}, the weights of the spin representation are
\begin{gather*}
 \pm {\eta_1\over 2} \pm \cdots \pm {\eta_{2m_1+2m_2}\over 2}\\
{}= \sum_{j_1\in I_1} {\theta_{j_1}+{\varphi_1+\varphi_2\over 2}\over 2} - \sum_{j_1\in \bar{I_1}} {\theta_{j_1}+{\varphi_1+\varphi_2\over 2}\over 2}
+ \sum_{j_1\in I_2} {\theta_{j_1}+{-\varphi_1-\varphi_2\over 2}\over 2} - \sum_{j_1\in \bar{I_2}} {\theta_{j_1}+{-\varphi_1-\varphi_2\over 2}\over 2} \\
\quad{} + \sum_{j_2\in I_1'} {\theta_{j_2}'+{-\varphi_1+\varphi_2\over 2}\over 2} - \sum_{j_2\in \bar{I_1'}} {\theta_{j_2}'+{-\varphi_1+\varphi_2\over 2}\over 2}
 + \sum_{j_2\in I_2'} {\theta_{j_2}'+{\varphi_1-\varphi_2\over 2}\over 2} - \sum_{j_2\in \bar{I_2'}} {\theta_{j_2}'+{\varphi_1-\varphi_2\over 2}\over 2} ,
\end{gather*}
where $I_1,I_2\subseteq \{1,\dots,m_1\}$ and $I_1',I_2'\subseteq \{1,\dots,m_2\}$.
\begin{itemize}\itemsep=0pt
\item The element $(-{\rm Id}_{2m_1},-{\rm Id}_{2m_2},-1)\in {\rm Sp}(m_1)\times {\rm Sp}(m_2)\times {\rm Spin}(4)$ corresponds to the parameters
\begin{gather*}
 \theta_{j_1} = \pi , \qquad \theta_{j_2}' = \pi , \qquad \varphi_1 = 2\pi,\qquad \varphi_2 = 0,
\end{gather*}
so that such a sum is equal to
\begin{gather*}\pi(2|I_1| - m_1 + 2|I_2'|-m_2) \end{gather*}
and the ef\/fect of $(-{\rm Id}_{2m_1},-{\rm Id}_{2m_2},-1)$ on each weight line is
\begin{gather*}e^{-i\pi(m_1+m_2)} = (-1)^{m_1+m_2}.
\end{gather*}

\item The element $({\rm Id}_{2m_1},-{\rm Id}_{2m_2},{\rm vol}_4)\in {\rm Sp}(m_1)\times {\rm Sp}(m_2)\times {\rm Spin}(4)$ corresponds to the parameters
\begin{gather*}
 \theta_{j_1} = 0 , \qquad \theta_{j_2}' = \pi , \qquad \varphi_1 = \pi,\qquad \varphi_2 = \pi,
\end{gather*}
so that the ef\/fect of $({\rm Id}_{2m_1},-{\rm Id}_{2m_2},{\rm vol}_4)$ on each weight line is
\begin{gather*}e^{-i\pi m_2} = (-1)^{m_2}.
\end{gather*}

\item The element $(-{\rm Id}_{2m_1},{\rm Id}_{2m_2},-{\rm vol}_4)\in {\rm Sp}(m_1)\times {\rm Sp}(m_2)\times {\rm Spin}(4)$ corresponds to the parameters
\begin{gather*}
 \theta_{j_1} = \pi , \qquad \theta_{j_2}' = 0 , \qquad \varphi_1 = \pi,\qquad
 \varphi_2 = -\pi,
\end{gather*}
so that the ef\/fect of $(-{\rm Id}_{2m_1},{\rm Id}_{2m_2},-{\rm vol}_4)$ on each weight line is
\begin{gather*}e^{-i\pi m_1} = (-1)^{m_1}.
\end{gather*}
\end{itemize}

Thus, in order to have a twisted Spin representation
\begin{gather*}\Delta_{4(m_1+m_2)}\otimes \ext^{u_1} \mathbb{C}^{2m_1}\otimes \ext^{u_2} \mathbb{C}^{2m_2}\otimes (\Delta_4^+)^{\otimes s}\otimes (\Delta_4^-)^{\otimes t}\end{gather*}
of \begin{gather*}{{\rm Sp}(m_1)\times {\rm Sp}(m_2)\times {\rm Spin}(4)\over \{\pm ({\rm Id}_{2m_1},{\rm Id}_{2m_2},1),\pm ({\rm Id}_{2m_1},-{\rm Id}_{2m_2},{\rm vol}_4)\}},\end{gather*} the exponents must satisfy
\begin{gather*}
m_1+u_1+t \equiv 0\,\,\,\mbox{(mod 2)} ,\qquad m_2+u_2+s \equiv 0\,\,\,\mbox{(mod 2)} .
\end{gather*}

\subsubsection[$r=6$]{$\boldsymbol{r=6}$}

By \eqref{eq: pesos r=2,6}, the weights of the spin representation are
\begin{gather*}
 \pm {\eta_1\over 2} \pm \cdots \pm {\eta_{8m}\over 2}\\
 {} = \sum_{j\in I_1} {\theta_{j}+{-\varphi_1+\varphi_2+\varphi_3\over 2}\over 2} - \sum_{j\in \bar{I_1}} {\theta_{j}+{-\varphi_1+\varphi_2+\varphi_3\over 2}\over 2}
 + \sum_{j\in I_2} {\theta_{j}+{\varphi_1-\varphi_2+\varphi_3\over 2}\over 2} - \sum_{j\in \bar{I_2}} {\theta_{j}+{\varphi_1-\varphi_2+\varphi_3\over 2}\over 2} \\
{} + \sum_{j\in I_3} {\theta_{j}+{\varphi_1+\varphi_2-\varphi_3\over 2}\over 2} - \sum_{j\in \bar{I_3}} {\theta_{j}+{\varphi_1+\varphi_2-\varphi_3\over 2}\over 2}
 + \sum_{j\in I_4} {\theta_{j}+{-\varphi_1-\varphi_2-\varphi_3\over 2}\over 2} - \sum_{j\in \bar{I_4}} {\theta_{j}+{-\varphi_1-\varphi_2-\varphi_3\over 2}\over 2} ,
\end{gather*}
where $I_1,I_2, I_3,I_4\subseteq \{1,\dots,m_1\}$.

\begin{itemize}\itemsep=0pt
\item The element $(-{\rm Id}_{m},-1)\in {\rm U}(m)\times {\rm Spin}(6)$ corresponds to the parameters
\begin{gather*}
 \theta_{j} = \pi , \qquad \varphi_1 = 2\pi,\qquad \varphi_2 = 0,\qquad \varphi_3 = 0,
\end{gather*}
so that its ef\/fect on each weight line is
\begin{gather*}e^{-2i\pi m} = 1.
\end{gather*}

\item The element $(i{\rm Id}_m,-{\rm vol}_6)\in {\rm U}(m)\times {\rm Spin}(6)$ corresponds to the parameters
\begin{gather*}
 \theta_{j} = {\pi\over 2} , \qquad \varphi_1 = -\pi,\qquad \varphi_2 = \pi, \qquad \varphi_3 = \pi,
\end{gather*}
so that its ef\/fect on each weight line is
\begin{gather*}e^{-i\pi m} = (-1)^{m}.
\end{gather*}

\end{itemize}
Thus, in order to have a twisted Spin representation
\begin{gather*}\Delta_{8m}\otimes \ext^{u_1} \mathbb{C}^{m}\otimes \ext^{u_2} \overline{\mathbb{C}^{m}}\otimes (\Delta_6^+)^{\otimes s}\otimes (\Delta_6^-)^{\otimes t}\end{gather*}
of ${{\rm U}(m) \times {\rm Spin}(6)\over \{\pm ({\rm Id}_{m},1),\pm (i{\rm Id}_{m},-{\rm vol}_6)\}}$, the exponents must satisfy
\begin{gather*}
 u_1+u_2+s+t \equiv 0 \,\,\,\mbox{(mod 2)},\qquad
 2m+u_1+3u_2+s+3t \equiv 0 \,\,\,\mbox{(mod 4)},\\
 2m+3u_1+u_2+3s+t \equiv 0 \,\,\,\mbox{(mod 4)}.
\end{gather*}

\subsubsection[$r=8$]{$\boldsymbol{r=8}$}
By \eqref{eq: pesos r=8 m1 odd m2 even}, if $m_1+1\equiv m_2\equiv 0$ (mod 2), the element $({\rm Id}_{m_1},-{\rm Id}_{m_2},{\rm vol}_8)\in {\rm SO}(m_1)\times {\rm SO}(m_2)\times {\rm Spin}(8)$ corresponds to the parameters
\begin{gather*}
 \theta_{j_1} = 0 , \qquad \theta_{j_2}' = \pi , \qquad \varphi_1 = \varphi_2 = \varphi_3 = \varphi_4 = \pi,
\end{gather*}
and its ef\/fect on each weight line is mutiplication by $-1$. Thus, we can have twisted Spin representations
\begin{gather*}\Delta_{8(m_1+m_2)}\otimes \ext^{u_1} \mathbb{C}^{m_1}\otimes \ext^{u_2} \mathbb{C}^{m_2}\otimes(\Delta_8^+)^{\otimes s}\otimes (\Delta_8^-)^{\otimes t}\end{gather*}
of{\samepage
\begin{gather*}
{{\rm SO}(m_1)\times {\rm SO}(m_2)\times {\rm Spin}(8)\over \{({\rm Id}_{m_1},{\rm Id}_{m_2},1),({\rm Id}_{m_1},-{\rm Id}_{m_2},{\rm vol}_8)\}}\end{gather*} if $u_2+t\equiv 1\,\,\,\mbox{(mod 2)}$ and $u_1,s\in \mathbb{N}$.}

Similarly, by \eqref{eq: pesos r=8 m1 even m2 odd}, if $m_1\equiv m_2+1\equiv 0$ (mod 2), we can have twisted Spin representations
\begin{gather*}\Delta_{8(m_1+m_2)}\otimes \ext^{u_1} \mathbb{C}^{m_1}\otimes \ext^{u_2} \mathbb{C}^{m_2}\otimes(\Delta_8^+)^{\otimes s}\otimes (\Delta_8^-)^{\otimes t} \end{gather*}
of
\begin{gather*}
{{\rm SO}(m_1)\times {\rm SO}(m_2)\times {\rm Spin}(8)\over \{({\rm Id}_{m_1},{\rm Id}_{m_2},1),(-{\rm Id}_{m_1},{\rm Id}_{m_2},-{\rm vol}_8)\}}
\end{gather*} if
$u_1+s\equiv 1\,\,\,\mbox{(mod 2)}$ and $u_2,t\in \mathbb{N}.$

\subsection{Twisting representations}\label{subsec: conditions 2}

For most $r$, almost-${\rm Cl}_r^0$-Hermitian manifolds are Spin \cite[Theorem~4.1]{Arizmendi-Garcia-Herrera}. In particular, this is the case when $r\geq 5$ and $r\not= 6,\,8$. Thus, we only need to choose suitable representations of the structure group $G$ to twist the spinor bundle:
\begin{itemize}\itemsep=0pt
 \item For $r\equiv 1,\,7$ (mod 8) our candidates are
 \begin{gather*}\ext^u \mathbb{C}^m\otimes \Delta_r^{\otimes s}, \qquad \sym^u\mathbb{C}^m\otimes \Delta_r^{\otimes s}.\end{gather*}
 They are representations of the structure group when
 \begin{gather*}\begin{array}{@{}ll}
u+s\equiv 0 \,\,\mbox{(mod 2)}& \text{if $m$ is even}, \\
u,s\in\mathbb{N} & \text{if $m$ is odd}.
 \end{array}
 \end{gather*}

 \item For $r\equiv 3,\,5$ (mod 8) our candidates are
 \begin{gather*}\ext^u \mathbb{C}^{2m}\otimes \Delta_r^{\otimes s},\qquad \sym^u\mathbb{C}^{2m}\otimes \Delta_r^{\otimes s}.\end{gather*}
 They are representations of the structure group when
\begin{gather*}u+s\equiv 0 \,\,\mbox{(mod 2)}.\end{gather*}

 \item For $r\equiv 4$ (mod 8) our candidates are
 \begin{gather*}\ext^{u_1} \mathbb{C}^{2m_1}\otimes \ext^{u_2} \mathbb{C}^{2m_2}\otimes (\Delta_r^+)^{\otimes s}\otimes (\Delta_r^-)^{\otimes t},
\\ \sym^{u_1}\mathbb{C}^{2m_1}\otimes \sym^{u_2}\mathbb{C}^{2m_2}\otimes (\Delta_r^+)^{\otimes s}\otimes (\Delta_r^-)^{\otimes t}.\end{gather*}
 They are representations of the structure group when
 \begin{gather*}
u_2+s \equiv 0 \,\,\,\mbox{(mod 2)},\qquad
u_1+t \equiv 0 \,\,\,\mbox{(mod 2)}.
 \end{gather*}

 \item For $r\equiv 2$ (mod 8), $r\not =2$, our candidates are
 \begin{gather*}\ext^{u_1} \mathbb{C}^{m}\otimes \ext^{u_2} \overline{\mathbb{C}^{m}}\otimes (\Delta_r^+)^{\otimes s}\otimes (\Delta_r^-)^{\otimes t},
 \qquad \sym^{u_1}\mathbb{C}^{m}\otimes \sym^{u_2}\overline{\mathbb{C}^{m}}\otimes (\Delta_r^+)^{\otimes s}\otimes (\Delta_r^-)^{\otimes t}.\end{gather*}
 They are representations of the structure group when
 \begin{gather*}
u_1+u_2+s+t \equiv 0 \,\,\,\mbox{(mod 2)},\qquad u_1+3u_2+3s+t \equiv 0 \,\,\,\mbox{(mod 4)},\\
3u_1+u_2+s+3t \equiv 0 \,\,\,\mbox{(mod 4)}.
 \end{gather*}

 \item For $r\equiv 6$ (mod 8) our candidates are
 \begin{gather*}\ext^{u_1} \mathbb{C}^{m}\otimes \ext^{u_2} \overline{\mathbb{C}^{m}}\otimes (\Delta_r^+)^{\otimes s}\otimes (\Delta_r^-)^{\otimes t},
 \qquad \sym^{u_1}\mathbb{C}^{m}\otimes \sym^{u_2}\overline{\mathbb{C}^{m}}\otimes (\Delta_r^+)^{\otimes s}\otimes (\Delta_r^-)^{\otimes t}.\end{gather*}
 They are representations of the structure group when
 \begin{gather*}
u_1+u_2+s+t \equiv 0 \,\,\,\mbox{(mod 2)},\qquad u_1+3u_2+s+3t \equiv 0 \,\,\,\mbox{(mod 4)},\\
3u_1+u_2+3s+t \equiv 0 \,\,\,\mbox{(mod 4)}.
 \end{gather*}

 \item For $r\equiv 0$ (mod 8) our candidates are
 \begin{gather*}\ext^{u_1} \mathbb{C}^{m_1}\otimes \ext^{u_2} \mathbb{C}^{m_2}\otimes (\Delta_r^+)^{\otimes s}\otimes (\Delta_r^-)^{\otimes t},\\
 \sym^{u_1}\mathbb{C}^{m_1}\otimes \sym^{u_2}\mathbb{C}^{m_2}\otimes (\Delta_r^+)^{\otimes s}\otimes (\Delta_r^-)^{\otimes t}.\end{gather*}
 They are representations of the structure group when
 \begin{alignat*}{3}
 & \begin{cases}
u_2+t\equiv 0 \,\,\,\mbox{(mod 2)} \\
u_1+s\equiv 0 \,\,\,\mbox{(mod 2)}
 \end{cases}
 \quad && \mbox{if $m_1\equiv m_2\equiv 0$ (mod 2)},& \\
& \begin{cases}
u_2+t\equiv 0 \,\,\,\mbox{(mod 2)} \\
u_1,s\in\mathbb{N}
 \end{cases}
 \quad & & \mbox{if $m_1+1\equiv m_2\equiv 0$ (mod 2)},& \\
& \begin{cases}
u_2,t\in\mathbb{N} \\
u_1+s\equiv 0 \,\,\,\mbox{(mod 2)}
 \end{cases}
 \quad && \mbox{if $m_1\equiv m_2+1\equiv 0$ (mod 2)},& \\
& u_1,u_2,s,t\in\mathbb{N}\quad && \mbox{if $m_1\equiv m_2\equiv 0$ (mod 2)}. &
 \end{alignat*}
\end{itemize}

\section{Index calculations}\label{sec: dirac-operators}

In this section, we recall the def\/inition of twisted Dirac operators, how to apply the Atiyah--Singer f\/ixed point formula~\cite{AS3}, (inf\/initesimal) automorphisms of almost-${\rm Cl}_r^0$-Hermitian manifolds and prove the vanishing Theorems~\ref{theo: vanishing 1},~\ref{theo: vanishing 2} and~\ref{theo: vanishing 3}.

\subsection{Rigidity of elliptic operators}

\begin{Definition}
Let $D\colon \Gamma(E) \longrightarrow \Gamma(F)$ be an elliptic operator acting on sections of the vector bundles $E$ and $F$ over a compact manifold $M$. The index of $D$ is the virtual vector space $\ind (D)=\ker(D)-\coker(D)$. If $M$ admits a circle action preserving $D$, i.e., such that $S^1$ acts on $E$ and $F$, and commutes with $D$, $\ind(D)$ admits a~Fourier decomposition into complex $1$-di\-men\-sional irreducible representations of $S^1$ $\ind(D) =\sum a_m L^m$, where $a_m\in \Z$ and $L^m$ is the representation of $S^1$ on $\C$ given by $\lambda \mapsto \lambda^m$. The elliptic operator~$D$ is called {\em rigid} if $a_m=0$ for all $m\neq 0$, i.e., $\ind(D)$ consists only of the trivial representation with multiplicity~$a_0$.
\end{Definition}

Let us recall three examples.

\begin{Example}
The deRham complex
\begin{gather*} d+d^* \colon \ \Omega^{\rm even} \longrightarrow \Omega^{\rm odd}\end{gather*}
from even-dimensional forms to odd-dimensional ones, where $d^*$ denotes the adjoint of the exterior derivative $d$, is {\em rigid} for any circle action on~$M$ by isometries since by Hodge theory the kernel and the cokernel of this operator consist of harmonic forms, which by homotopy invariance stay f\/ixed under the circle action.
\end{Example}

\begin{Example} The signature operator on an oriented manifold
\begin{gather*} d_s \colon \ \Omega_c^{+} \longrightarrow \Omega_c^{-}\end{gather*}
from even to odd complex forms under the Hodge $*$ operator is {\em rigid} for any circle action on~$M$ by isometries since the kernel and cokernel of this operator consist of harmonic forms.
\end{Example}

\begin{Example} The Dirac operator on a Spin manifold is {\em rigid} for any circle action by iso\-met\-ries~\cite{AH}.
\end{Example}

\subsection{Twisted Dirac operators}

In this subsection, let $M$ be a $4n$-dimensional oriented Riemannian manifold. $M$ is Spin if its orthonormal frame bundle $P_{{\rm SO}(4n)}$ admits a~double cover by a principal bundle $P_{{\rm Spin}(4n)}$ with structure group ${\rm Spin}(4n)$, which gives rise to the spinor bundle
\begin{gather*}P_{{\rm Spin}(4n)}\times_{\kappa}\Delta_{4n}. \end{gather*}
The Levi-Civita connection on $P_{{\rm SO}(4n)}$ can be lifted to $P_{{\rm Spin}(4n)}$ to def\/ine a covariant dif\/ferentia\-tion~$\nabla$ on~$\Delta$
\begin{gather*}\nabla\colon \ \Gamma (\Delta ) \longrightarrow \Gamma (T^*\otimes \Delta),\end{gather*}
and the (elliptic and self-adjoint) Dirac operator
\begin{gather*}\dirac (\psi) = \sum_{i=1}^{4n} e^i\cdot \nabla_{e_i}\psi\end{gather*}
for $\psi\in\Gamma(\Delta)$, where $(e_1,\dots, e_{4n})$ is a local orthonormal frame. Since the spin representation decomposes, the Dirac operator can be split into two parts
\begin{gather*}\dirac\colon \ \Gamma (\Delta_+ ) \longrightarrow \Gamma ( \Delta_-),\qquad
\dirac^*\colon \ \Gamma (\Delta_- ) \longrightarrow \Gamma ( \Delta_+).\end{gather*}

We are interested in Dirac operators with coef\/f\/icients in auxiliary vector bundles $F$ equipped with a~covariant derivative $\nabla^F\colon \Gamma(F)\longrightarrow\Gamma(T^*\otimes F)$. The Dirac operator twisted by $F$ (or with coef\/f\/icients
in $F$)
\begin{gather*}(\dirac\otimes F)\colon \ \Gamma (\Delta_+\otimes F) \longrightarrow \Gamma (\Delta_-\otimes F)\end{gather*}
is def\/ined by
\begin{gather*}(\dirac\otimes F) (\psi\otimes f) = \left(\sum_{i=1}^{4n} e^i\cdot \nabla_{e_i}\psi\right)\otimes f + \sum_{i=0}^{4n}
\mu\big(e^i\otimes \psi\big) \otimes \nabla_{e_i}^F f,\end{gather*}
where $\psi\in\Gamma(\Delta)$, $f\in\Gamma(F)$.

\begin{Remark}
If the manifold is not Spin, there may exist well def\/ined twisted spinor bundles (as above), as it happens when the structure group of~$M$ reduces to a subgroup of ${\rm SO}(4n)$ and $\Delta_{4n}\otimes F$ is a~representation of such subgroup.
\end{Remark}

\subsection{Index formula and localization}\label{subsec: localization}

Let $M$ be an compact $4n$-dimensional oriented Riemannian manifold. Let us assume that the bundle
\begin{gather*}\Delta_{4n}\otimes F\end{gather*}
is well def\/ined, where we will use the same symbol to denote the representation and the associated vector bundle, where the $\dim_{\mathbb{C}}(F)=p$. Since $\Delta_{4n}\otimes F$ is a Clif\/ford bundle, by the Atiyah--Singer index theorem \cite{Berline,Roe}, the index of the twisted Dirac operators can be computed as
\begin{gather*}\ind(\dirac\otimes F) = \big\langle \A(M) \ch(F),[M]\big\rangle,\end{gather*}
where $\ch(\cdot)$ denotes the Chern character, $\A(M)$ denotes the \A-genus, and $[M]$ denotes the fundamental cycle of $M$. In terms of formal roots,
\begin{gather*}
c(TM\otimes\mathbb{C})= (1+\eta_1)(1-\eta_1)\cdots (1+\eta_{2n})(1-\eta_{2n}),\\
 p(TM)= \big(1+\eta_1^2\big)\cdots \big(1+\eta_{2n}^2\big),\\
 c(F) = (1+\nu_1)\cdots (1+\nu_{p}),\\
 \ch(F)= \sum_{l=1} e^{\nu_l}, \\
 \ind(\dirac\otimes F) = \left\langle \sum_{l=1} e^{\nu_l} \cdot \prod_{i=1}^{2n}
{ \eta_i \over e^{\eta_i\over2}-e^{-{\eta_i\over2}}},[M]\right\rangle .
\end{gather*}
If $M$ admits a non-trivial $S^1$ action that lifts to $\Delta_{4n}\otimes F$, the equivariant version of the index can be written in terms of the local data of the $S^1$-f\/ixed point set~$M^{S^1}$. More precisely, let $z\in S^1$ be a generic element of $S^1$. By the Atiyah--Singer f\/ixed point theorem~\cite{AS3,Berline}
\begin{gather*}\ind(\dirac\otimes F)_{z}= \sum_{P\subset M^{S^1}} \mu(P,z),\end{gather*}
where $\mu(P,z)$ is the local contribution of the oriented f\/ixed point submanifold $P\subset M^{S^1}$, which can be computed as follows. The $S^1$ action on $M$ induces a~decomposition of $TM$ over~$P$,
\begin{gather}
TM\vert_{P}= \sum_{k} N_k,\label{eq: decomposition1}
\end{gather}
where $N_k$ is a bundle over $P$ whose f\/ibers are representations of $S^1$ on which $z\in S^1$ acts as an automorphism with multiple eigenvalue~$z^k$, $k\in\mathbb{Z}$. Note that $P$ inherits an orientation since~$M$ is oriented and the bundles $N_k$ for $k\not=0$ are naturally oriented. Formally, by means of the splitting principle, we can write
\begin{gather}
 TM\vert_{P}= \mathcal{L}^{q_1}+\dots + \mathcal{L}^{q_{2n}},\label{eq: decomposition2}
\end{gather}
where $\mathcal{L}$ corresponds to the standard representation of $S^1$ on $\mathbb{C}$, so that $z\in S^1$ acts by multiplication by $z^{q_i}$ on~$\mathcal{L}^{q_i}$. The integers $q_i=q_i(P)\in \mathbb{Z}$ are the exponents of the action at~$P$, which correspond to the aforementioned numbers~$k$. Thus, following \cite[p.~67]{Hirzebruch},
\begin{gather*}\mu(P,z)= \left\langle \sum z^{-n_k}e^{\nu_k} \prod_{q_i=0}
{ \eta_i \over e^{\eta_i\over2}-e^{-{\eta_i\over2}}} \prod_{q_j\not = 0}
{1 \over z^{-{q_j\over2}}e^{\eta_j\over2}-z^{q_j\over2}e^{-{\eta_j\over2}}},[P]\right\rangle ,
\end{gather*}
where $n_k=n_k(P)$ are the exponents of the action on $F$ restricted to $P$. The function $\mu(P,z)$ is a rational function of the complex variable $z$ with zeroes at $0$ and $\infty$ as long as
\begin{gather}
|n_k|< {1\over 2}\left(|q_1(P)|+\dots + |q_{2n}(P)|\right)\label{eq: fundamental inequality}
\end{gather}
for all $1\leq k\leq p$. If such a condition is fulf\/illed for all $P\subset M^{S^1}$, then $\ind(\dirac\otimes F)_z$ is a rational function of $z$ with zeroes at $0$ and $\infty$. Notice that $\ind(\dirac\otimes F)_z$ also belongs to the representation ring $R(S^1)$ of $S^1$, which can be identif\/ied with the Laurent polynomial ring $\mathbb{Z}[z, z^{-1}]$. Hence, by Lemma~\ref{lemma: useful lemma},
\begin{gather*}\ind(\dirac\otimes F)=\ind(\dirac\otimes F)_1 =0,\end{gather*}
i.e.,
\begin{gather*}\big\langle \A(M)\ch(F),[M]\big\rangle =0.\end{gather*}

\subsection[$\widehat A$-genus of almost-${\rm Cl}_r^0$-Hermitian manifolds]{$\boldsymbol{\widehat A}$-genus of almost-$\boldsymbol{{\rm Cl}_r^0}$-Hermitian manifolds}

Given a $4n$-dimensional Riemannian manifold, according to the splitting principle with respect to the maxinal torus of ${\rm SO}(4n)$, its complexif\/ied tangent bundle splits formally as follows
\begin{gather*}TM\otimes \mathbb{C} = L_1\oplus L_1^{-1} \oplus\cdots\oplus L_{2n}\oplus L_{2n}^{-1}\end{gather*}
and, therefore,
\begin{gather*}
c(TM\otimes \mathbb{C}) = (1+x_1)(1-x_1)\cdots(1+x_{2n})(1-x_{2n})= \big(1-x_1^2\big)\cdots\big(1-x_{2n}^2\big)
\end{gather*}
and its Pontrjagin class is
\begin{gather*}
p(TM) = \big(1+x_1^2\big)\cdots \big(1+x_{2n}^2\big),
\end{gather*}
and the \A-genus is given by
\begin{gather*}
\A(M) =\prod_{j=1}^{2n} {x_i/2\over \sinh(x_i/2)} =\prod_{j=1}^{2n} {x_i\over e^{x_i/2}-e^{-x_i/2}}.
\end{gather*}

In the following, we will set $x_i=\eta_i$ from Section \ref{subsec: weights}.

\subsubsection[$r\equiv 1,\,7 \,\,{\rm (mod\,\,\,8)}$]{$\boldsymbol{r\equiv 1,\,7 \,\,{\rm (mod\,\,\,8)}}$}

The \A-genus is given by
\begin{gather*}
\A(M)
=\prod_{j=1}^{[{m\over 2}]}\prod_{k=1}^{2^{[{r\over 2}]}}{\theta_j + \lambda_k\over e^{\theta_j + \lambda_k\over 2}
-e^{-{\theta_j + \lambda_k\over 2}}}
\end{gather*}
if $m$ is even, and
\begin{gather*}
\A(M)
=\prod_{j=1}^{[{m\over 2}]}\prod_{k=1}^{2^{[{r\over 2}]}}{\theta_j + \lambda_k\over e^{\theta_j + \lambda_k\over 2}
-e^{-{\theta_j + \lambda_k\over 2}}}
\prod_{l=1}^{2^{[{r\over 2}]-1}} {\lambda_l\over e^{\lambda_l\over 2}
-e^{-{\lambda_l\over 2}}}
\end{gather*}
if $m$ is odd.

\subsubsection[$r\equiv 2,\,6 \,\,{\rm (mod\,\,\,8)}$]{$\boldsymbol{r\equiv 2,\,6 \,\,{\rm (mod\,\,\,8)}}$}

The \A-genus is given by
\begin{gather*}
\A(M)
= \begin{cases}
\displaystyle \prod _{j=1}^{m}\prod_{k=1}^{2^{{r\over 2}-1}}{\theta_j + \lambda_k^+\over e^{\theta_j + \lambda_k^+\over 2}
-e^{-{\theta_j + \lambda_k^+\over 2}}} & \mbox{if $r\equiv 2$ (mod 8)},\\
\displaystyle \prod_{j=1}^{m}\prod_{k=1}^{2^{{r\over 2}-1}}{\theta_j + \lambda_k^-\over e^{\theta_j + \lambda_k^-\over 2}
-e^{-{\theta_j + \lambda_k^-\over 2}}} & \mbox{if $r\equiv 6$ (mod 8)}.
\end{cases}
\end{gather*}

\subsubsection[$r\equiv3,\,5 \,\,{\rm (mod\,\,\,8)}$]{$\boldsymbol{r\equiv3,\,5 \,\,{\rm (mod\,\,\,8)}}$}

The \A-genus is given by
\begin{gather*}
\A(M) =\prod_{j=}^{m}\prod_{k=1}^{2^{[{r\over 2}]}}{\theta_j + \lambda_k\over e^{\theta_j + \lambda_k\over 2}
-e^{-{\theta_j + \lambda_k\over 2}}}.
\end{gather*}

\subsubsection[$r\equiv 4 \,\,{\rm (mod\,\,\,8)}$]{$\boldsymbol{r\equiv 4 \,\,{\rm (mod\,\,\,8)}}$}

The \A-genus is given by
\begin{gather*}
\A(M)=\prod_{j_1=1}^{m_1}\prod_{k=1}^{2^{{r\over 2}-1}}{\theta_{j_1} + \lambda_k^+\over e^{\theta_{j_1} + \lambda_k^+\over 2}
-e^{-{\theta_{j_1} + \lambda_k^+\over 2}}}
\prod_{j_2=1}^{m_2}\prod_{k=1}^{2^{{r\over 2}-1}}{\theta_{j_2}' + \lambda_{k}^-\over e^{\theta_{j_2}'
+ \lambda_{k}^-\over 2}-e^{-{\theta_{j_2}' + \lambda_{k}^-\over 2}}}.
\end{gather*}

\subsubsection[$r\equiv 0 \,\,{\rm (mod\,\,\,8)}$]{$\boldsymbol{r\equiv 0 \,\,{\rm (mod\,\,\,8)}}$}
We can set
\begin{itemize}\itemsep=0pt
\item if $m_1$, $m_2$ are even,
\begin{gather*}
\A(M)
=\prod_{j_1=1}^{{m_1\over2}}\prod_{k=1}^{2^{{r\over 2}-1}}{\theta_{j_1}
+ \lambda_k^+\over e^{\theta_{j_1} + \lambda_k^+\over 2}
-e^{-{\theta_{j_1 }+ \lambda_k^+\over 2}}}
\prod_{j_2=0}^{{m_2\over 2}}\prod_{k=1}^{2^{{r\over 2}-1}}{\theta_{j_2}'
+ \lambda_{k}^-\over e^{\theta_{j_2}'
+ \lambda_{k}^-\over 2}
-e^{-{\theta_{j_2}' + \lambda_{k}^-\over 2}}};
\end{gather*}

\item if $m_1$ is even and $m_2$ is odd,
\begin{gather*}
\A(M)= \prod_{j_1=1}^{{m_1\over2}}\prod_{k=1}^{2^{{r\over 2}-1}}{\theta_{j_1}
+ \lambda_k^+\over e^{\theta_{j_1} + \lambda_k^+\over 2}
-e^{-{\theta_{j_1 }+ \lambda_k^+\over 2}}}
\prod_{j_2=1}^{[{m_2\over 2}]}\prod_{k=1}^{2^{{r\over 2}-1}}{\theta_{j_2}'
+ \lambda_{k}^-\over e^{\theta_{j_2}'
+ \lambda_{k}^-\over 2}
-e^{-{\theta_{j_2}' + \lambda_{k}^-\over 2}}}
\prod_{l=1}^{2^{{r\over 2}-2}} {\lambda_{l}^-\over e^{\lambda_{l}^-\over 2}
-e^{-{\lambda_{l}^-\over 2}}} ;
\end{gather*}

 \item if $m_1$ is odd and $m_2$ is even,
\begin{gather*}
\A(M) =\prod_{j_1=1}^{[{m_1\over2}]}\prod_{k=1}^{2^{{r\over 2}-1}}{\theta_{j_1}
+ \lambda_k^+\over e^{\theta_{j_1} + \lambda_k^+\over 2}
-e^{-{\theta_{j_1 }+ \lambda_k^+\over 2}}}
\prod_{l=1}^{2^{{r\over 2}-2}} {\lambda_{l}^+\over e^{\lambda_{l}^+\over 2}
-e^{-{\lambda_{l}^+\over 2}}}
\prod_{j_2=1}^{{m_2\over 2}}\prod_{k=1}^{2^{{r\over 2}-1}}{\theta_{j_2}'
+ \lambda_{k}^-\over e^{\theta_{j_2}'
+ \lambda_{k}^-\over 2}
-e^{-{\theta_{j_2}' + \lambda_{k}^-\over 2}}} ;
\end{gather*}

 \item if $m_1$, $m_2$ are odd,
\begin{gather*}
\A(M)
=\prod_{j_1=1}^{[{m_1\over2}]}\prod_{k=1}^{2^{{r\over 2}-1}}{\theta_{j_1}
+ \lambda_k^+\over e^{\theta_{j_1} + \lambda_k^+\over 2}
-e^{-{\theta_{j_1 }+ \lambda_k^+\over 2}}} \prod_{l=1}^{2^{{r\over 2}-2}} {\lambda_{l}^+\over e^{\lambda_{l}^+\over 2}
-e^{-{\lambda_{l}^+\over 2}}}\\
\hphantom{\A(M)=}{} \times
 \prod_{j_2=1}^{[{m_2\over 2}]}\prod_{k=1}^{2^{{r\over 2}-1}}{\theta_{j_2}'
+ \lambda_{k}^-\over e^{\theta_{j_2}'+ \lambda_{k}^-\over 2}-e^{-{\theta_{j_2}' + \lambda_{k}^-\over 2}}}\prod_{l=1}^{2^{{r\over 2}-2}} {\lambda_{l}^-\over e^{\lambda_{l}^-\over 2} -e^{-{\lambda_{l}^-\over 2}}} .
\end{gather*}
\end{itemize}

\subsection{Inf\/initesimal automorphisms}\label{subsec: infinitesimal automorphism}
An {\em automorphism} of an almost-${\rm Cl}_r^0$-Hermitian manifold $M$ is an isometry which preserves the almost even-Clif\/ford Hermitian structure. A vector f\/ield $X$ on $M$ is an {\em infinitesimal automorphism} if it is a Killing vector f\/ield that preserves the structure, i.e., locally
\begin{gather*}\mathcal{L}_X J_{ij} =\sum_{k<l} \alpha_{kl}^{(ij)} J_{kl},\end{gather*}
for some (local) functions $\alpha_{kl}^{(ij)}$, where $\mathcal{L}_X$ denotes the Lie derivative in the direction of~$X$. Consider
\begin{gather*}
{\mathcal L}_X(J_{ij}(Y)) = ({\mathcal L}_XJ_{ij})(Y) + J_{ij}({\mathcal L}_XY),
\end{gather*}
which can be written in terms of the Levi-Civita connection $\nabla$ as follows
\begin{gather*}
\nabla_X(J_{ij}(Y))-\nabla_{J_{ij}(Y)}X = \sum_{k<l} \alpha_{kl}^{(ij)} J_{kl}(Y)+J_{ij}(\nabla_XY-\nabla_YX).
\end{gather*}
Now, if $p\in M$ is such that $X_p=0$,
\begin{gather*}
-\nabla_{J_{ij}(Y)}X = \sum_{k<l} \alpha_{kl}^{(ij)} J_{kl}(Y)-J_{ij}(\nabla_YX),
\end{gather*}
i.e.,
\begin{gather*}
[J_{ij},\nabla X](Y) = \sum_{k<l} \alpha_{kl}^{(ij)} J_{kl}(Y).
\end{gather*}
Hence, $(\nabla X)_p$ is a skew-symmetric endomorphism such that
\begin{gather*}
[J_{ij},\nabla X] = \sum_{k<l} \alpha_{kl}^{(ij)} J_{kl}.
\end{gather*}
i.e., $(\nabla X)_p$ belongs to Lie algebra of the structure group of $M$ \cite{Arizmendi-Garcia-Herrera, Arizmendi-Herrera}.

We will say that a smooth circle action on an almost-${\rm Cl}_r^0$-Hermitian manifold is an {\em action by automorphisms} if the corresponding Killing vector f\/ield is an inf\/initesimal automorphism.

\begin{Example}
The $16$-dimensional symmetric space
\begin{gather*}{F_4\over {\rm Spin}(9)}\end{gather*}
has an almost-${\rm Cl}_9^0$-Hermitian structure admitting $S^1$ actions by automorphisms~\cite{Friedrich-weak}.
\end{Example}

\subsection[Exponents of the $S^1$ action]{Exponents of the $\boldsymbol{S^1}$ action}
In this section, let $M$ be a compact, rank $r$ almost even-Clif\/ford Hermitian manifold with a~non-trivial (ef\/fective) $S^1$ action by automorphisms. Let $P\subset M^{S^1}$ be an $S^1$-f\/ixed submanifold. The corresponding inf\/initesimal isometry $X$ is such that $(\nabla X)_p \in\mathfrak{so}(N)$ at any f\/ixed point $p\in P$. This corresponds to the induced action of~$S^1$ to~$T_pM$, and such a circle lies in a maximal torus. The tangent space at $p$ decomposes as in Section~\ref{subsec: localization}. In fact, we can now be more precise about these exponents. By Section~\ref{subsec: infinitesimal automorphism}, $(\nabla X)_p$ belongs to a Cartan subalgebra of the Lie algebra of the structure group, and we can assume that the decomposition \eqref{eq: decomposition1} is compatible with a~decomposition such as~\eqref{eq: decomposition2} into complex lines with respect to a~maximal torus of such a~group. Hence, we can read of\/f the exponents of the action with respect to the weights given in Section~\ref{subsec: weights}:
\begin{gather*}
\begin{array}{|c|c|c|c|}
 \hline
 r \mbox{\ {\rm (mod 8)} } & \pm q_i & &
\\
\hline
 0 & {t_{j_1}\pm h_k^+\over 2},\, {t_{j_2}'\pm h_k^-\over 2}& 1\leq j_1\leq [{m_1\over 2}], & 1\leq k\leq 2^{[{r\over 2}]-2}\tsep{4pt}
\\
& \big( {h_l\over2}, \, {h_{2^{{r\over 2}-1}+l}\over2} \big) & 1\leq j_2\leq [{m_2\over 2}] & \big(1\leq l\leq 2^{[{r\over 2}]-3}\big)\bsep{3pt}\\
 \hline
 1,\, 7 & {t_j\pm h_k\over 2} \ \big( {h_l\over2} \big) & 1\leq j\leq [{m\over 2}] & 1\leq k\leq 2^{[{r\over 2}]-1} \tsep{4pt}\\
 &&& \big(1\leq l\leq 2^{[{r\over 2}]-2}\big)\bsep{3pt}\\
\hline
 2,\, 6 & {t_j \pm h_k^+\over 2},\, {-t_j \pm h_k^-\over 2} & 1\leq j\leq m & 1\leq k\leq 2^{[{r\over 2}]-2}\tsep{4pt} \bsep{3pt}\\
\hline
 3,\, 5 & {t_j\pm h_k\over 2} & 1\leq j\leq m & 1\leq k\leq 2^{[{r\over 2}]-1}\tsep{4pt} \bsep{3pt} \\
\hline
 4 & {t_{j_1}\pm h_k^+\over 2}, \, {t_{j_2}'\pm h_k^-\over 2} & 1\leq j_1\leq m_1, \, 1\leq j_2\leq m_2 & 1\leq k\leq 2^{[{r\over 2}]-2}\tsep{4pt} \bsep{3pt}\\
\hline
 \end{array}
\end{gather*}
where $m$, $m_1$, $m_2$ denote the corresponding multiplicities. Here, the numbers ${t_j\over 2}$, ${t_j'\over 2}$ are the exponents corresponding to the complex representations $E$, $E_1$, $E_2$ described in Section~\ref{subsubsection: structure groups}, $f_j$ are the exponents for the ${\rm SO}(r)$ representation for $r$ odd or $\mathbb{P}{\rm SO}(r)$ representation for $r$ even, $h_k$ denote the numbers
\begin{gather*} \pm f_1\pm\cdots \pm f_{[{r\over 2}]}\end{gather*}
in some order for $r$ odd, and $h_k^\pm$ denote the numbers
\begin{gather*} \pm f_1\pm\cdots \pm f_{r\over 2}\end{gather*}
with an even or odd number of negative signs respectively, listed in some order for $r$ even.

\subsection{Vanishing theorems}

In this section, we give the main details of the proofs of the vanishing theorems.

\begin{Theorem}\label{theo: vanishing 1} Let $M$ be a compact $N$-dimensional almost-${\rm Cl}_r^0$-Hermitian admitting a smooth circle action by automorphisms, $r\geq 3$. Let $E$, $E_1$, $E_2$ be the $($locally defined$)$ bundles described in~\eqref{eq: complexified tangent space}, $m$, $m_1$, $m_2$ the corresponding multiplicities and $u$, $u_1$, $u_2$, $s$, $t$ be non-negative integers satisfying the conditions given in Sections~{\rm \ref{subsec: conditions 1}} and {\rm \ref{subsec: conditions 2}}. Then,
\begin{itemize}\itemsep=0pt
\item for $r\equiv 1,\, 7 \,\,({\rm mod\,\,\,} 8)$, if $0\leq u+s<[\frac{m}{2}]$,
\begin{gather*}\big\langle \ch(\ext^u E)\ch(\Delta_r)^s\A(M),[M]\big\rangle =0;\end{gather*}

\item for $r\equiv 3,\, 5 \,\,({\rm mod\,\,\,} 8)$, if $0\leq u+s<m$,
 \begin{gather*}\big\langle \ch(\ext^u E)\ch(\Delta_r)^s\A(M),[M]\big\rangle =0;\end{gather*}

\item for $r\equiv 0 \,\,({\rm mod\,\,\,} 8)$, if $0\leq u_1+s<[\frac{m_1}{2}]$, $0\leq u_2+t<[\frac{m_2}{2}]$,
\begin{gather*}\big\langle \ch(\ext^{u_1} E_1)\ch(\ext^{u_2} E_2)\ch(\Delta_r^+)^{s}\ch(\Delta_r^-)^{t}\A(M),[M]\big\rangle =0;\end{gather*}

\item for $r\equiv 2,\, 6 \,\,({\rm mod\,\,\,} 8)$, if $0\leq u_1+s<m$ and $0 \leq u_2+t <m$, or if $0\leq u_1+t<m$ and $0 \leq u_2+s <m$,
\begin{gather*}\big\langle \ch(\ext^{u_1} E)\ch(\ext^{u_2} \overline{E})\ch(\Delta_r^+)^{s}\ch(\Delta_r^-)^{t}\A(M),[M]\big\rangle =0;\end{gather*}

\item for $r\equiv 4 \,\,({\rm mod\,\,\,} 8)$, if $0\leq u_1+s<m_1$ and $0\leq u_2+t<m_2$,
\begin{gather*}\big\langle \ch(\ext^{u_1} E_1)\ch(\ext^{u_2} E_2)\ch(\Delta_r^+)^{s}\ch(\Delta_r^-)^{t}\A(M),[M]\big\rangle =0.\end{gather*}
\end{itemize}
If the inequalities are not strict, the indices are rigid.
\end{Theorem}

\begin{proof}
Since the $S^1$ action is by automorphisms of the almost even-Clif\/ford Hermitian structure, the action lifts to the bundles associated to the structure group, such as the twisted spin bundles we are considering. Given that the arguments are similar in all cases, we will only describe the calculation for $r\equiv 1,\, 7$ (mod~8) and $m$ even.

Let $P\subset M^{S^1}$ be an $S^1$-f\/ixed submanifold. By Section~\ref{subsec: infinitesimal automorphism}, over~$P$ the circle group of automorphisms maps non-trivially to the structure group ${\rm SO}(m){\rm Spin}(r)$, so that the f\/ibers of the bundles $\Delta_N\otimes \ext^uE\otimes \Delta_r^{\otimes s}$ over points of~$P$ decompose as sums of representations of~$S^1$. Recall that
\begin{gather*}
 \ch\big(\ext^uE\big) = \sum_{1\leq i_1<\dots<i_u\leq 2[\frac{m}{2}]}e^{\vartheta_{i_1} +\dots+\vartheta_{i_u}},
\end{gather*}
where
\begin{gather*}
 \vartheta_j = \theta_j,\qquad \vartheta_{[m/2]+j} = -\theta_j,\qquad j=1,\dots,[m/2].
\end{gather*}

Thus, the exponents of the twist will be of the form
\begin{gather*}{1\over 2}\left(\sum_{a=1}^c (-1)^{\varepsilon_a}t_{i_a}+\sum_{b=1}^s (-1)^{\delta_b}h_{l_b}\right),\end{gather*}
where $0\leq c\leq u$, $\varepsilon_a,\delta_b\in \{ 0,1 \}$. There are two points to verify in the proof: f\/irstly, that the contributions~$\mu(P,z)$ are rational functions and, secondly, that the exponents of the twisting bundles and the tangent space satisfy the inequality~\eqref{eq: fundamental inequality}.

The f\/irst one follows from the fact that the f\/ibers of the bundles $\Delta_N\otimes \ext^uE\otimes \Delta_r^{\otimes s}$ over $P$ decompose as sums of representations of $S^1$. Formally, according to the splitting principle, if
\begin{gather*}
TM_c= L_1\oplus L_1^{-1}\oplus\cdots\oplus L_{N/2}\oplus L_{N/2}^{-1},
\end{gather*}
then
\begin{gather*}
\Delta_N= \big(L_1^{1/2}\oplus L_1^{-1/2}\big)\otimes\cdots \otimes\big(L_{N/2}^{1/2}\oplus L_{N/2}^{-1/2}\big)\\
\hphantom{\Delta_N}{} = L_1^{1/2}\otimes\cdots\otimes L_{N/2}^{1/2}\oplus\cdots\oplus L_1^{-1/2}\otimes\cdots\otimes L_{N/2}^{-1/2},
\end{gather*}
so that the $S^1$-exponents on these lines will be of the form
\begin{gather*}\sum \left(\pm{t_j\pm h_k\over 4}\right) + \sum \left(\pm{h_l\over 4}\right).\end{gather*}
The bundle $\Delta_N\otimes \ext^uE\otimes \Delta_r^{\otimes s}$ will have integer exponents over $P$ of the form
\begin{gather*}
{1\over 2}\left(\sum_{a=1}^c (-1)^{\varepsilon_a}t_{i_a}+\sum_{b=1}^s (-1)^{\delta_b}h_{l_b}\right)
+\sum_{q_j\not=0} \left((-1)^{\gamma_j}{t_j+ h_k\over 4}\right) \\
\qquad{} +\sum_{q_{j'}\not=0} \left((-1)^{\gamma_{j'}}{t_{'j}- h_k\over 4}\right)
+ \sum_{q_l\not=0} \left((-1)^{\zeta_l}{h_l\over 4}\right).
\end{gather*}
Thus, the powers of $z$ in each summand of $\mu(P,z)$ can be rearranged in order to show that such a summand is a product of rational functions
such as the one described in Lemma~\ref{lemma: useful lemma}.

For the second point, it is suf\/f\/icient to consider the exponents of the form
 \begin{gather*}{1\over 2}\left(\sum_{a=1}^u (-1)^{\varepsilon_a}t_{i_a}\pm sh_k\right).\end{gather*}
Since $u+s<[\frac{m}{2}]$, there exists an $s$-tuple of indices $j_1<\dots < j_s$ such that $\{j_1,\dots, j_s\}\subset \{1,\dots,[\frac{m}{2}]\}-\{i_1,\dots,i_u\}$. Thus,
 \begin{gather*}
 \left\vert\sum_{a=1}^u (-1)^{\varepsilon_a}t_{i_a}\pm sh_k\right\vert =
 \left\vert\sum_{a=1}^u (-1)^{\varepsilon_a}\left({t_{i_a}+h_k\over 2}+{t_{i_a}-h_k\over 2}\right)
 \pm \sum_{b=1}^s \left({h_k+t_{j_b}\over 2}+{h_k-t_{j_b}\over 2}\right) \right\vert\\
 \hphantom{\left\vert\sum_{a=1}^u (-1)^{\varepsilon_a}t_{i_a}\pm sh_k\right\vert}{}
 \leq\sum_{a=1}^u
 \left(\left\vert{t_{i_a}+h_k\over 2}\right\vert+\left\vert{t_{i_a}-h_k\over 2}\right\vert\right)
 + \sum_{b=1}^s \left(\left\vert{h_k+t_{j_b}\over 2}\right\vert+\left\vert{h_k-t_{j_b}\over 2}\right\vert\right)\\
\hphantom{\left\vert\sum_{a=1}^u (-1)^{\varepsilon_a}t_{i_a}\pm sh_k\right\vert}{} \leq\sum_{i=1}^m
 \left|{t_{i}+h_k\over 2}\right|+\left|{t_{i}-h_k\over 2}\right|
<\sum_{l=1}^{2^{[r/2]-1}}\sum_{i=1}^m
 \left|{t_{i}+h_l\over 2}\right|+\left|{t_{i}-h_l\over 2}\right|\\
\hphantom{\left\vert\sum_{a=1}^u (-1)^{\varepsilon_a}t_{i_a}\pm sh_k\right\vert}{} \leq \sum_{c=1}^{N/2} |q_c|,
 \end{gather*}
which is the corresponding version of the inequality \eqref{eq: fundamental inequality} in Section~\ref{subsec: localization}
\end{proof}

\begin{Theorem}\label{theo: vanishing 2} Let $M$ be a compact $N$-dimensional almost-${\rm Cl}_r^0$-Hermitian admitting a smooth circle action by automorphisms, $r\geq 3$. Let $E$, $E_1$, $E_2$ be the $($locally defined$)$ bundles described in~\eqref{eq: complexified tangent space}, $m$, $m_1$, $m_2$ the corresponding multiplicities and $u$, $u_1$, $u_2$, $s$, $t$ be non-negative integers satisfying the conditions given in Sections~{\rm \ref{subsec: conditions 1}} and~{\rm \ref{subsec: conditions 2}}. Then,
\begin{itemize}\itemsep=0pt
\item for $r\equiv 1,\, 7 \,\,({\rm mod\,\,\,} 8)$, if $0\leq u+s<[\frac{m}{2}]$ and $u\leq 2^{[{r\over 2}]-1}$,
 \begin{gather*}\big\langle \ch(\sym^u E)\ch(\Delta_r)^s\A(M),[M]\big\rangle =0;\end{gather*}

\item for $r\equiv 3,\, 5 \,\,({\rm mod\,\,\,} 8)$, if $0\leq u+s<m$ and $u\leq 2^{[{r\over 2}]-1}$,
 \begin{gather*}\big\langle \ch(\sym^u E)\ch(\Delta_r)^s\A(M),[M]\big\rangle =0;\end{gather*}

\item for $r\equiv 0 \,\,({\rm mod\,\,\,} 8)$, if $0\leq u_1+s<[\frac{m_1}{2}]$, $0\leq u_2+t<[\frac{m_2}{2}]$ and $u_1,u_2\leq 2^{[{r\over 2}]-2}$,
 \begin{gather*}\big\langle \ch(\sym^{u_1} E_1)\ch(\sym^{u_2} E_2)\ch(\Delta_r^+)^{s}\ch(\Delta_r^-)^{t}\A(M),[M]\big\rangle =0;\end{gather*}

\item for $r\equiv 2,\, 6 \,\,({\rm mod\,\,\,} 8)$, if $u_1,u_2\leq 2^{[{r\over 2}]-2}$ and one of $0\leq u_1+s<m$, $0 \leq u_2+t <m$ or $0\leq u_1+t<m$, $0 \leq u_2+s <m$,
 \begin{gather*}\big\langle \ch(\sym^{u_1} E)\ch(\sym^{u_2} \overline{E})\ch(\Delta_r^+)^{s}\ch(\Delta_r^-)^{t}\A(M),[M]\big\rangle =0;\end{gather*}

\item for $r\equiv 4 \,\,({\rm mod\,\,\,} 8)$, if $0\leq u_1+s<m_1$, $0\leq u_2+t<m_2$ and $u_1,u_2\leq 2^{[{r\over 2}]-2}$,
 \begin{gather*}\big\langle \ch(\sym^{u_1} E_1)\ch(\sym^{u_2} E_2)\ch(\Delta_r^+)^{s}\ch(\Delta_r^-)^{t}\A(M),[M]\big\rangle =0.\end{gather*}
\end{itemize}
If the inequalities are not strict, the indices are rigid.
\end{Theorem}

\begin{proof}
We will only describe the relevant changes to the calculation for $r\equiv 1,\, 7$ (mod 8) and~$m$ even. Let $P\subset M^{S^1}$ be an $S^1$-f\/ixed submanifold. Recall that
\begin{gather*}
 \ch\big(S^uE\big) = \sum_{1\leq i_1\leq\dots\leq i_u\leq 2[\frac{m}{2}]}e^{\vartheta_{i_1} +\dots+\vartheta_{i_u}},
\end{gather*}
where
\begin{gather*}
 \vartheta_j = \theta_j,\qquad \vartheta_{[m/2]+j} = -\theta_j, \qquad j=1,\dots,[m/2].
\end{gather*}

Thus, the exponents of the twist will be of the form
 \begin{gather*}{1\over 2}\left(\sum_{a=1}^u (-1)^{\varepsilon_a}t_{i_a}+\sum_{b=1}^s (-1)^{\delta_b}h_{l_b}\right),\end{gather*}
where $\varepsilon_a,\delta_b\in \{ 0,1 \}$. It is suf\/f\/icient to consider the exponents of the form
\begin{gather*}{1\over 2}\left(\sum_{a=1}^u (-1)^{\varepsilon_a}t_{i_a}\pm sh_k\right).\end{gather*}
Among them, there are two extreme types, namely the ones equal to exponents of the exterior powers which we already know how to deal with, and the ones such as $u t_1$. For such an exponent, consider
\begin{gather*}
 |ut_1| = \left|\sum_{l=1}^u{t_1+h_l\over 2} + {t_1-h_l\over 2}\right|
 \leq \sum_{l=1}^u\left|{t_1+h_l\over 2}\right| + \left|{t_1-h_l\over 2}\right|
\leq \sum_{l=1}^{2^{[{r\over 2}]-1}}\left|{t_1+h_l\over 2}\right| + \left|{t_1-h_l\over 2}\right|\\
\hphantom{|ut_1|}{} < \sum_{j=1}^{[{m\over 2}]}\sum_{l=1}^{2^{[{r\over 2}]-1}}\left|{t_1+h_l\over 2}\right| + \left|{t_1-h_l\over 2}\right|
\leq \sum_{c=1}^{N/2} |q_c|,
\end{gather*}
if $u<2^{[r/2]-1}$.
\end{proof}

\begin{Theorem}\label{theo: vanishing 3} Let $M$ be a compact $N$-dimensional almost-${\rm Cl}_r^0$-Hermitian admitting a smooth circle action by automorphisms, $r\geq 3$. Let $E$, $E_1$, $E_2$ be the $($locally defined$)$ bundles described in~\eqref{eq: complexified tangent space}, $m$, $m_1$, $m_2$ the corresponding multiplicities and $u_i$, $v_i$, $u'_i$, $v'_i$, $s$, $t$ be non-negative integers satisfying analogous conditions to those given in Sections~{\rm \ref{subsec: conditions 1}} and~{\rm \ref{subsec: conditions 2}}. Then,
 \begin{itemize}\itemsep=0pt
\item for $r\equiv 1,\,7 \,\,({\rm mod\,\,\,} 8)$, if
 \begin{gather*}
 0\leq \sum_{i=1}^b u_i+\sum_{j=1}^b v_j+s<\left[\frac{m}{2}\right]
\end{gather*}
and
\begin{gather*}
 a + \sum_{i=1}^b v_i\leq 2^{[{r\over 2}]-1}, \\
 \bigg\langle \ch\bigg(\bigotimes_{i=1}^a\ext^{u_i}E\otimes \bigotimes_{j=1}^b\sym^{v_j} E\otimes (\Delta_r)^{\otimes s}\bigg)\A(M),[M]\bigg\rangle =0;\end{gather*}

\item for $r\equiv 3,\,5 \,\,({\rm mod\,\,\,} 8)$, if
 \begin{gather*}
 0\leq \sum_{i=1}^b u_i+\sum_{j=1}^b v_j+s<m
 \end{gather*}
 and
 \begin{gather*}
 a + \sum_{i=1}^b v_i\leq 2^{[{r\over 2}]-1}, \\
 \bigg\langle \ch\bigg(\bigotimes_{i=1}^a \ext^{u_i}E\otimes\bigotimes_{j=1}^b \sym^{v_j} E\otimes(\Delta_r)^{\otimes s}\bigg)\A(M),[M]\bigg\rangle =0;\end{gather*}

\item for $r\equiv 0 \,\,({\rm mod\,\,\,} 8)$, if
 \begin{gather*}
 0\leq \sum_{i=1}^b u_i+\sum_{j=1}^b v_j+s <\left[\frac{m_1}{2}\right], \qquad 0\leq \sum_{i=1}^c u'_i+\sum_{j=1}^d v'_j+t<\left[\frac{m_2}{2}\right]
 \end{gather*} and
 \begin{gather*}
 a + \sum_{i=1}^b v_i, c + \sum_{i=1}^d v'_i\leq 2^{[{r\over 2}]-2}, \\
 \bigg\langle \ch\bigg(\bigotimes_{i=1}^a\ext^{u_i}E_1\otimes\bigotimes_{j=1}^b\sym^{v_j} E_1 \otimes
 \bigotimes_{k=1}^c\ext^{u'_k}E_2\otimes\bigotimes_{l=1}^d\sym^{v'_l} E_2 \\
 \qquad{} \otimes (\Delta_r^+)^{\otimes s}\otimes(\Delta_r^-)^{\otimes t} \bigg)
 \A(M),[M]\bigg\rangle =0;\end{gather*}

 \item for $r\equiv 2 \,\,({\rm mod\,\,\,} 8)$, if
 \begin{gather*}
 a + \sum_{i=1}^b v_i,\qquad c + \sum_{i=1}^d v'_i \leq 2^{[{r\over 2}]-2},\end{gather*}
and
 \begin{gather*}
 0\leq \sum_{i=1}^b u_i+\sum_{j=1}^b v_j+s<m,\qquad 0 \leq \sum_{i=1}^c u'_i+\sum_{j=1}^d v'_j+t <m
 \end{gather*}
 or
 \begin{gather*}
 0\leq \sum_{i=1}^b u_i+\sum_{j=1}^b v_j+t<m,\qquad 0 \leq \sum_{i=1}^c u'_i+\sum_{j=1}^d v'_j+s <m, \\
 \bigg\langle \ch\bigg(\bigotimes_{i=1}^a\ext^{u_i}E\otimes\bigotimes_{j=1}^b\sym^{v_j} E
 \otimes\bigotimes_{k=1}^c \ext^{u'_k}\overline{E}\otimes\bigotimes_{l=1}^d\sym^{v'_l} \overline{E} \\
 \qquad{} \otimes (\Delta_r^+)^{\otimes s}\otimes (\Delta_r^-)^{\otimes t} \bigg)
 \A(M),[M]\bigg\rangle =0;\end{gather*}

 \item for $r\equiv 4 \,\,({\rm mod\,\,\,} 8)$, if
 \begin{gather*}
 0\leq \sum_{i=1}^b u_i+\sum_{j=1}^b v_j+s<m_1,\qquad 0\leq \sum_{i=1}^c u'_i+\sum_{j=1}^d v'_j+t <m_2
 \end{gather*}
 and
 \begin{gather*}
 a + \sum_{i=1}^b v_i,\qquad c + \sum_{i=1}^d v'_i\leq 2^{[{r\over 2}]-2},
 \\
 \bigg\langle \ch\bigg(\bigotimes_{i=1}^a\ext^{u_i}E_1\otimes\bigotimes_{j=1}^b\sym^{v_j} E_1 \otimes
 \bigotimes_{k=1}^c\ext^{u'_k}E_2\otimes\bigotimes_{l=1}^d\sym^{v'_l} E_2 \\
 \qquad{} \otimes(\Delta_r^+)^{\otimes s}\otimes(\Delta_r^-)^{\otimes t} \bigg) \A(M),[M]\bigg\rangle =0.\end{gather*}
 \end{itemize}
 If the inequalities are not strict, the indices are rigid.
 \end{Theorem}

\begin{Remark}When $r=3$, Theorems \ref{theo: vanishing 1} and \ref{theo: vanishing 2} return the vanishings for almost quaternion-Hermitian manifolds proved in \cite{Herrera-Herrera}.
\end{Remark}

\begin{Remark} Theorems \ref{theo: vanishing 1}, \ref{theo: vanishing 2} and \ref{theo: vanishing 3} do not restrict to the well known vanishings for almost Hermitian manifolds proved in \cite{Hattori}, which require a divisibility condition on $c_1(M)$. This is due to the fact that the structure group of a $2m$-dimensional almost Hermitian manifold is ${\rm U}(m)$ instead of
\begin{gather*}{{\rm U}(m)\times {\rm Spin}(2)\over \{\pm({\rm Id}_2,1),\pm(i{\rm Id}_2,-{\rm vol}_2)\}}.\end{gather*}
\end{Remark}

\begin{Remark}For $r\not = 3,\,4,\,6,\,8$, an almost-${\rm Cl}_r^0$-Hermitian manifold is Spin (see \cite[Theo\-rem~4.1]{Arizmendi-Garcia-Herrera}). Thus, for $u=u_1=u_2=s=t=0$, the vanishings in the theorems restrict to Atiyah--Hirzebruch's vanishing.
\end{Remark}

\subsection*{Acknowledgements}

The f\/irst named author was supported by CONACyT. The second named author was partially supported by a CONACyT grant. The second named author wishes to thank the International Centre for Theoretical Physics and the Institut des Hautes \'Etudes Scientif\/iques for their hospitality and support. We would like to express our gratitude to the anonymous referees for their careful reading of this manuscript and their helpful comments.

\pdfbookmark[1]{References}{ref}
\LastPageEnding

\end{document}